\documentclass[11pt, letterpaper]{amsart}

\usepackage{amsmath}
\usepackage{amssymb}
\usepackage{amsfonts}
\usepackage{ dsfont }
\usepackage{cases}
\usepackage{cite}
\usepackage[hyphens]{url}
\usepackage{hyperref}
\usepackage{amsthm}
\usepackage{upgreek}
\usepackage{color}

\numberwithin{equation}{section}
\DeclareMathOperator*{\esssup}{ess\,sup}
\DeclareMathOperator*{\essinf}{ess\,inf}

\DeclareMathOperator*{\argmin}{arg\,min}

\newtheorem{theorem}{Theorem}[section]

\newtheorem{lemma}[theorem]{Lemma}
\newtheorem{proposition}[theorem]{Proposition}
\theoremstyle{definition}
\newtheorem{definition}[theorem]{Definition}
\theoremstyle{remark}
\newtheorem{remark}[theorem]{Remark}

\newcommand{\leref}{Lemma~\ref}
\newcommand{\deref}{Definition~\ref}

\newcommand{\prref}{Proposition~\ref}
\newcommand{\thref}{Theorem~\ref}

\title[]{On controller-stopper problems with jumps and their applications to indifference pricing of American options}\thanks{We would like to thank the anonymous referees and the AE for their insightful comments which helped us improve our paper. }
\author[]{Erhan Bayraktar} \thanks{This research was supported in part by the National Science Foundation under grants DMS 0955463 and DMS 1118673.}  
\address{Department of Mathematics, University of Michigan}
\email{erhan@umich.edu}
\author[]{Zhou Zhou}
\address{Department of Mathematics, University of Michigan}
\email{zhouzhou@umich.edu}
\date{\today}
\begin{document}
\maketitle
\begin{abstract}
We consider controller-stopper problems in which the controlled processes can have jumps. The global filtration is represented by the Brownian filtration, enlarged by the filtration generated by the jump process. We assume that there exists a conditional probability density function for the jump times and marks given the filtration of the Brownian motion and decompose the global controller-stopper problem into controller-stopper problems with respect to the Brownian filtration, which are determined by a backward induction. We apply our decomposition method to indifference pricing of American options under multiple default risk. The backward induction leads to a system of reflected backward stochastic differential equations (RBSDEs). We show that there exists a solution to this RBSDE system and that the solution provides a characterization of the value function.
\end{abstract}
\section{Introduction}
 The problem of pricing American options and the very closely related stochastic control problem of a controller and stopper either cooperating or playing a zero-sum game has been analyzed extensively for continuous processes. In particular,  \cite{Karatzas1} considers the super-hedging problem; \cite{MR2178425}, \cite{MR2172784}, \cite{MR2435857}, and \cite{MR2928345}  consider the controller-stopper problems, and \cite{Sircar1} resolves the  indifference pricing problem using the results of \cite{MR2178425}. We will consider the above problems in the presence of jumps in the state variables. 
 
The stochastic control problems in the above setup can be solved by Hamilton-Jacobi-Bellman integro-differential equations in the Markovian setup, or by Reflected Backward Stochastic Differential Equations (RBSDEs) with jumps, generalizing the results of \cite{2012arXiv1201.2690J}, which we will call the \emph{global approach}. 
We prefer to use an alternative approach in which we convert 
the problem with jumps into a sequence of problems without jumps \`{a} la \cite{erhan3}, which uses this result for  linear pricing of American options, and \cite{pham1} which uses this approach to solve indifference pricing problems for European-style optimal control problems with jumps under a conditional density hypothesis. 

One may wonder what the local approach we propose brings over the global approach in financial applications.
Indeed, in the second part of the paper, where we give an application of the decomposition results of controller-stopper games to indifference pricing of American options, one may use the methods in \cite{Schweizer1} and \cite{Sircar 1} to convert the original problem into a dual problem over martingale measures which could be represented as a solution of an
 RBSDE with jumps or integro-PDEs for a non-linear free boundary problem. Compared to this global approach, what we propose has several advantages:
 
 (a) Our method tells us how to behave optimally between jumps. For instance, our stopping times are not hitting times. They are hitting times of certain levels between jumps. But these levels change as the jumps occur. This tells us how the investor reacts to defaults and changes her stopping strategies. However, the global method can provide little insight into the impact of jumps on the optimal strategies.

(b) Like in \cite{pham2} and \cite{pham1}, our decomposition approach allows us to formulate the optimal investment problems where the portfolio constraint set can be updated after each default time, depending on the past defaults, which is financially relevant. Nevertheless, in the global approach the admissible set of strategies has to be fixed in the beginning.

(c) The decomposition result is useful in the analysis of Backward stochastic differential equations (BSDEs) with jumps. For example, \cite{Kharroubi1} uses the decomposition result of \cite{pham1} to construct a solution to BSDEs with jumps. Similar decomposition results were used earlier by \cite{MR1283589} in understanding the structure of control problems in a piece-wise deterministic setting. Also, see \cite{MR2260062} for example for the application of the decomposition idea to the solution of a quickest change detection problem.


Following the setup in  \cite{pham2} and \cite{pham1} we also assume that there are at most $n$ jumps. Assuming the number of jumps is finite is not restrictive for financial modeling purposes. We think of jumps representing default events. The jumps in our framework have both predictable and totally inaccessible parts. That is, we are in the hybrid default modeling framework considered by \cite{El2}, \cite{Pham4} and \cite{pham1} and following these papers we make the assumption that the joint distribution of jump times and marks has a conditional density. For a more precise formulation see the standing assumption in Section~\ref{sec:expectation}.

In this jump-diffusion model, we give a decomposition of the controller-stopper problem into controller-stopper problems with respect to the Brownian filtration, which are determined by a backward induction. We apply this decomposition method to indifference pricing of American options under multiple jump risk, extending the results of \cite{pham1}. The solution of this problem leads to a system of reflected backward stochastic differential equations (RBSDEs). We show that there exists a solution to this RBSDE system and the solution provides a characterization of the value function, which can be thought of as an extension of \cite{Ying1}.

Our first result, see Theorem~\ref{theorem 3-1} and Proposition~\ref{prop 3-3}, is a decomposition result for stopping times of the global filtration (the filtration generated by the Brownian motion and jump times and marks).
 Next, in Section~\ref{sec:expectation}, 
 we show that the expectation of an optional process with jumps can be computed by a backward induction, where each step is an expectation with respect to the Brownian filtration. In Section~\ref{sec:control}, we consider the controller-stopper problems with jumps and decompose the original problem into controller-stopper problems with respect to the Brownian filtration. Finally, we apply our decomposition result to obtain the indifference buying/selling price of American options with jump/default risk in Section~\ref{sec:indifference-pricing} and characterize the optimal trading strategies and the optimal stopping times in Theorem~\ref{theorem 4-4} and Theorem~\ref{theorem 4-6}, which resolves a saddle point problem, which is an important and difficult problem in the controller-stopper games.

Since we work with optional processes (because our optimization problem contains a state variable with unpredictable jumps), we can not directly rely on the decomposition result of \cite{Jeulin1} in Lemma 4.4 and Remark 4.5, or the corresponding result in\cite{MR519998} (which is for predictable processes and the filtrations involved are right-continuous) from the classical theory of enlargement of filtrations. (See also Chapter 6 of \cite{MR2020294} for an exposition of this theory in English.)  It is well known in the theory of enlargement of filtrations that for a right-continuous enlargement, a decomposition for optional process is not true in general; the remark on page 318 of \cite{MR509204} gives a counter example. See also the introduction of the recent paper by \cite{Song1}. This is because in the case of optional processes the monotone class argument used in the proof of Lemma 4.4 in  \cite{Jeulin1} does not work for the right-continuously enlarged filtration. The phenomenon described here is in fact a classical example demonstrating the well-known exchangeability problem between intersection and the supremum of $\sigma$-algebras.  In our problem we work in an enlarged filtration which is not right-continuous. This allows to get optional decomposition results with respect to the enlarged filtration. On the other hand, since the enlargement is not right-continuous, no classical stochastic calculus tools can be used to solve the problem anymore. Therefore, our approach gives an important contribution to the stochastic optimization literature.  Also, as opposed to  \cite{Jeulin1} we consider a progressive enlargement with several jumps and jump marks. On the other hand, our decomposition of the controller-stopper problems into control-stopper problems in the smaller filtration can be viewed as a non-linear extension of the classical decomposition formulas due to Jeulin \cite{Jeulin1}.

In the rest of this section we will introduce the probabilistic setup and notation that we will use in the rest of the paper.

\subsection{Probabilistic setup}
As in \cite{pham1}, we start with $(\Omega,\mathbb{F},\mathbb{P})$ corresponding to the jump-free probability space, where $\mathbb{F}=(\mathcal{F}_t)_{t=0}^\infty$ is the filtration generated by the Brownian motion, satisfying the usual conditions. We assume that there are at most $n$ jumps. Define $\Delta_0=\emptyset$ and
$$\Delta_k=\left\{(\theta_1, \dotso,\theta_k): \ 0\leq\theta_1\dotso\leq\theta_k\right\},\ \ \ k=1,\dotso, n,$$
which represents the space of first $k$ jump times. For $k=1,\dotso, n$, let $e_k$ be the $k$-th jump mark taking values in some Borel subset $E$ of $\mathbb{R}^{\hat d}$.  For $k=0,\dotso,n$, let $\mathcal{D}^{k}$ be the filtration generated by the first $k$ jump times and marks, i.e.,
$$\mathcal{D}^{k}_t=\vee_{i=1}^{k}\big(\sigma(1_{\{\zeta_i\leq s\}},\ s\leq t)\vee\sigma(\ell_i1_{\{\zeta_i\leq s\}},\  s\leq t)\big).$$
Let
$$\mathcal{G}^k=\mathcal{F}\vee\mathcal{D}^k, \ \ k=0,\dotso,n.$$
Denote by $\mathbb{G}^k=(\mathcal{G}_t^k)_{t=0}^\infty$ for $k=0,\dotso,n$, and $\mathbb{G}=\mathbb{G}^n$. (One should note that these filtrations are not necessarily right continuous. When we look at the supremum of two $\sigma$ algebras, the resulting $\sigma$ algebra does not have to be right continuous. This is due to the famous exchangeability problem between the intersection and the supremum of two $\sigma$ algebras.) Then $(\Omega, \mathbb{G}^k, \mathbb{P})$ is the probability space including at most the first $k$ jumps, $k=0,\dotso,n$. Let $(\Omega, \mathbb{G}, \mathbb{P})=(\Omega, \mathbb{G}^n, \mathbb{P})$ which we refer to as the global probability space. Note that for $k=0,\dotso,n$, we may characterize each element in $\Omega$ as $(\omega_1,\theta_1,\dotso,\theta_k,e_1,\dotso,e_k)$, when the random variable we consider is $\mathcal{G}_{\infty}^k$-measurable, where $\omega_1$ is viewed as the Brownian motion argument and $\mathcal{G}_{\infty}^k=\cup_{t=0}^{\infty}\mathcal{G}_t^k$, see page 76 in \cite{texterhan}.

Next we will introduce some notation that will be used in the rest of the paper. 

\subsection{Notation}

\begin{itemize}
\item For any $(\theta_1,\dotso,\theta_k)\in \Delta_k, \ (\ell_1,\dotso,\ell_k)\in E^k$, we denote by
$$\pmb\theta_k=(\theta_1,\dotso,\theta_k),\ \ \ \pmb\ell_k=(\ell_1,\dotso,\ell_k),\ \ \ \ \  k=1,\dotso,n.$$
We also denote by $\pmb\zeta_k=(\zeta_1,\dotso,\zeta_k)$, and $\pmb\ell_k=(\ell_1,\dotso, \ell_k)$. From now on, for $k=1,\dotso,n$, we use $\theta_k,\pmb\theta_k,e_k,\pmb e_k$ to represent given fixed numbers or vectors, and $\zeta_k,\pmb\zeta_k,\ell_k,\pmb\ell_k$ to represent random jump times or marks.  
\item $\mathcal{P}_\mathbb{F}$ is the $\sigma$-algebra of $\mathbb{F}$-predictable measurable subsets on $\mathbb{R}_+\times\Omega$, i.e., the $\sigma$-algebra generated by the left-continuous $\mathbb{F}$-adapted processes. 
\item $\mathcal{P}_{\mathbb{F}}(\Delta_k,E^k)$ is the set of indexed $\mathbb{F}$-predictable processes $Z^k(\cdot)$, i.e.,  the map $(t,\omega,\pmb\theta_k,\pmb\ell_k)\rightarrow Z_t^k(\omega,\pmb\theta_k,\pmb\ell_k)$ is $\mathcal{P}_\mathbb{F}\otimes\mathcal{B}(\Delta_k)\otimes\mathcal{B}(E^k)$-measurable, for $k=1,\dotso,n$. We also denote $\mathcal{P}_\mathbb{F}$ as $\mathcal{P}_{\mathbb{F}}(\Delta_0,E^0)$.
\item $\mathcal{O}_\mathbb{F}$(resp.$\mathcal{O}_\mathbb{G}$) is the $\sigma$-algebra of $\mathbb{F}$(resp.$\mathbb{G}$)-optional measurable subsets on $\mathbb{R}_+\times\Omega$, i.e., the $\sigma$-algebra generated by the right-continuous $\mathbb{F}$(resp.$\mathbb{G}$)-adapted processes.
\item $\mathcal{O}_{\mathbb{F}}(\Delta_k,E^k)$ is the set of indexed $\mathbb{F}$-adapted processes $Z^k(\cdot)$, i.e.,  the map $(t,\omega,\pmb\theta_k,\pmb\ell_k)\rightarrow Z_t^k(\omega,\pmb\theta_k,\pmb\ell_k)$ is $\mathcal{O}_\mathbb{F}\otimes\mathcal{B}(\Delta_k)\otimes\mathcal{B}(E^k)$-measurable, for $k=1,\dotso,n$.  We also denote $\mathcal{O}_\mathbb{F}$ as $\mathcal{O}_{\mathbb{F}}(\Delta_0,E^0)$.
\item For any $\mathcal{G}_{\infty}^k$-measurable random variable $X$, we sometimes denote it as  $X=X(\omega_1,\pmb\zeta_k,\pmb\ell_k)=X(\pmb\zeta_k,\pmb\ell_k)$. Given $\pmb\zeta_k=\pmb\theta_k,\ \pmb\ell_k=\pmb e_k$, we denote $X$ as $X=X(\omega_1,\pmb\theta_k,\pmb\ell_k)=X(\pmb\theta_k,\pmb\ell_k)$. Similar notations apply for any $\mathbb{G}^k$-adapted process $(Z_t)_{t\geq 0}$ and its stopped version $Z_{\tau}$, where $\tau$ is a $\mathbb{G}^k$-stopping time.
\item For $T\in[0,\infty]$, $\Delta_k(T):=\Delta_k\cap[0,T]^k$.
\item$\displaystyle \mathcal{S}_c^{\infty}[t,T]:=\Big\{Y:\ \mathbb{F}\text{-adapted continuous,}\ ||Y||_{\mathcal{S}_c^\infty[t,T]}:=\displaystyle\esssup_{(s,\omega)\in[t,T]\times\Omega}|Y_s(\omega)|<\infty\Big\}.$
\item$\displaystyle \mathcal{S}_c^\infty(\Delta_k(T),E^k):=\Big\{Y^k\in\mathcal{O}_{\mathbb{F}}(\Delta_k,E^k):\ Y^k \text{ is continuous, and}\ ||Y^k||_{\mathcal{S}_c^\infty(\Delta_k(T),E^k)}:=$ \\
\text{ }\ \ \ \ \ \ \ \ \ \ \ \ \ \ \ \ \ \ \ \ $\displaystyle\sup_{(\pmb\theta_k,\pmb e_k)\in\Delta_k(T)\times E^k}||Y^k(\pmb\theta_k,\pmb e_k)||_{\mathcal{S}_c^\infty[\theta_k,T]}<\infty\Big\}, \ \ \ k=0,\dotso,n.$
\item${\bf L}_W^2[t,T]:=\Big\{Z:\ \mathbb{F}\text{-predictable,} \ \mathbb{E}\left[\int_t^T|Z_s|^2ds\right]<\infty\Big\}.$
\item${\bf L}_W^2(\Delta_k(T),E^k):=\Big\{Z^k\in\mathcal{P}_{\mathbb{F}}(\Delta_k,E^k):\ \mathbb{E}\left[\int_{\theta_k}^T|Z_t^k(\pmb\theta_k,\pmb e_k)|^2dt\right]<\infty,\ \ \forall(\pmb\theta_k,\pmb e_k)\in$\\
\text{ }\ \ \ \ \ \ \ \ \ \ \ \ \ \ \ \ \ \ \ \ $\displaystyle\Delta_k(T)\times E^k\Big\},\ \ \ k=0,\dotso,n.$
\item$\displaystyle\mathbf{A}[t,T]:=\big\{K:\ \mathbb{F}\text{-adapted continuous increasing,}\ \ K_t=0, \ \mathbb{E}K_T^2<\infty\big\},\ k=0,\dotso,n.$
\item$\displaystyle\mathbf{A}(\Delta_k(T),E^k):= \big\{K^k: \ \forall (\pmb{\theta}_k,\pmb{e}_k)\in\Delta_k(T)\times E^k,\ \ K^k(\pmb{\theta}_k,\pmb{e}_k) \ \in\mathbf{A}[\theta_k,T]\big\},\ k=0,\dotso,n.$
\item We use \textbf{eq}$(H,f)_{s\leq t\leq T}$ to represent the RBSDE
\begin{numcases}{}
Y_t=H_T-\int_t^T f(r,Y_r,Z_r)dr+\int_t^TZ_rdW_r+(K_T-K_t), \ \ \ \ \  s\leq t\leq T, \notag\\
 Y_t\geq H_t,\ \ \ \ \ s\leq t\leq T, \notag\\
  \int_s^T(Y_t- H_t)dK_t=0,  \notag
  \end{numcases}
and \textbf{EQ}$(\mathcal{H},\mathfrak{f})_{s\leq t\leq T}$ to represent the RBSDE
\begin{numcases}{}
\mathcal{Y}_t=\mathcal{H}_T+\int_t^T \mathfrak{f}(r,\mathcal{Y}_r,\mathcal{Z}_r)dr-\int_t^T\mathcal{Z}_rdW_r+(\mathcal{K}_T-\mathcal{K}_t), \ \ \ \ \  s\leq t\leq T, \notag\\
 \mathcal{Y}_t\geq \mathcal{H}_t,\ \ \ \ \ s\leq t\leq T, \notag\\
  \int_s^T(\mathcal{Y}_t- \mathcal{H}_t)d\mathcal{K}_t=0. \notag
  \end{numcases}
\end{itemize}

\section{Decomposition of $\mathbb{G}$-stopping times}

\thref{theorem 3-1} and Proposition~\ref{prop 3-3}, on the decomposition $\mathbb{G}$-stopping times, are the main results of this section. 

\begin{theorem}\label{theorem 3-1}
 $\tau$ is a $\mathbb{G}$-stopping time if and only if it has the decomposition:
\begin{equation}\label{2-3} 
\tau=\tau^0 1_{\{\tau^0<\zeta_1\}}+\sum_{k=1}^{n-1}\tau^k(\pmb\zeta_k,\pmb\ell_k) 1_{\{\tau^0\geq\zeta_1\}\cap\dotso\cap\{\tau^{k-1}\geq\zeta_k\}\cap\{\tau^k <\zeta_{k+1}\}}+\tau^n(\pmb\zeta_n,\pmb\ell_n) 1_{\{\tau^0\geq\zeta_1\}\dotso\cap\{\tau^{n-1}\geq\zeta_n\}},\end{equation}
for some $(\tau^0,\dotso,\tau^n)$, where $\tau^0$ is an $\mathbb{F}$-stopping time, and $\tau^k(\pmb\zeta_k,\pmb\ell_k)$ is a $\mathbb{G}^k$-stopping time satisfying 
\begin{equation}\label{eq:tgtjt}
\tau^k(\pmb\zeta_k,\pmb\ell_k)\geq\zeta_k, \; k=1,\dotso,n.
\end{equation}
\end{theorem}
\begin{proof} If $\tau$ has the decomposition \eqref{2-3}, then
\begin{eqnarray}
\{\tau\leq t\}&=&\Big(\{\tau^0<\zeta_1\}\cap\{\tau^0\leq t\}\Big)\bigcup_{k=1}^{n-1}\Big(\{\tau^0\geq\zeta_1\}\cap\dotso\cap\{\tau^{k-1}\geq\zeta_k\}\cap\{\tau^k<\zeta_{k+1}\}\cap\{\tau^k\leq t\}\Big)\notag\\
&&\cup\Big(\{\tau^0\geq\zeta_1\}\cap\dotso\{\tau^{n-1}\geq\zeta_n\}\cap\{\tau^n\leq t\}\Big).\notag
\end{eqnarray}
For $k=1,\dotso,n$, since $\{\tau^k<\zeta_{k+1}\}\in\mathcal{G}_{\tau^k}$, and
$$\{\tau^{i-1}\geq\zeta_i\}\in\mathcal{G}_{\zeta_i}\subset\mathcal{G}_{\zeta_k}\subset\mathcal{G}_{\tau^k},\ \ i=1,\dotso,k,$$
we have
$$\{\tau^0\geq\zeta_1\}\cap\dotso\cap\{\tau^{k-1}\geq\zeta_k\}\cap\{\tau^k<\zeta_{k+1}\}\cap\{\tau^k\leq t\}\notag\in\ \mathcal{G}_t.$$
Similarly we can show $\{\tau^0<\zeta_1\}\cap\{\tau^0\leq t\}\in\mathcal{G}_t$ and 
$$\{\tau^0\geq\zeta_1\}\cap\dotso\{\tau^{n-1}\geq\zeta_n\}\cap\{\tau^n\leq t\}\ \in\mathcal{G}_t.$$
If $\tau$ is a $\mathbb{G}$-stopping time, we will proceed in 3 steps to show that it has the decomposition \eqref{2-3}.
\medskip

\noindent\textbf{Step 1}: We will show that for any discretely valued $\mathbb{G}$-stopping time
$$\tau=\sum_{1\leq i\leq\infty}a_i1_{A_i},$$
where  $0\leq a_1<a_2<\dotso< a_{\infty}=\infty$ and $(A_i\in\mathcal{G}_{a_i})_{1\leq i\leq\infty}$ is a partition of $\Omega$, there exists a $\mathbb{G}^k$-stopping time $\tau^k=\tau^k(\pmb\zeta_k,\pmb\ell_k)$, such that
\begin{equation}\label{2-5}\tau 1_{\{\tau<\zeta_{k+1}\}}=\tau^k 1_{\{\tau<\zeta_{k+1}\}} \ \ \ \text{and}\ \ \ \{\tau<\zeta_{k+1}\}=\{\tau^k<\zeta_{k+1}\},\end{equation}
for $k=0,\dotso,n-1$. First, we have
$$\{\tau<\zeta_{k+1}\}=\bigcup_{1\leq i\leq\infty}\big(\{\tau<\zeta_{k+1}\}\cap\{A_i\}\big)=\bigcup_{1\leq i\leq\infty}\big(\{a_i<\zeta_{k+1}\}\cap\{A_i\}\big).$$
To complete Step 1, we need the following lemma:
\begin{lemma}\label{lemma 2-4}
For $i=1,\dotso,\infty$, and $A_i \in \mathcal{G}_{a_i}$, there exists $\tilde A_i\in\mathcal{G}_{a_i}^k$, such that
\begin{equation}
\label{2-4}\{a_i<\zeta_{k+1}\}\cap\tilde A_i=\{a_i<\zeta_{k+1}\}\cap A_i.
\end{equation}
Moreover, $(\tilde A_i)_{1\leq i\leq\infty}$ can be chosen to be mutually disjoint.
\end{lemma}
\noindent \emph{Proof of Lemma~\ref{lemma 2-4}}. 
Since for $j\geq k+1$,
\begin{eqnarray}
&&\big(\sigma(1_{\{\zeta_j\leq s\}},\ s\leq a_i)\vee\sigma(\ell_j1_{\{\zeta_j\leq s\}},\  s\leq a_i)\big)\cap\{a_i<\zeta_{k+1}\}\notag\\
&&=\sigma\Big(\{\zeta_j\leq s\},\ \big(\{\ell\in C\}\cap\{\zeta_j\leq t\}\big)\cup\{\zeta_j>t\},\ \ s,t\leq a_i, \ C\in\mathcal{B}(E)\Big)\cap\{a_i<\zeta_{k+1}\}\notag\\
&&=\sigma\bigg(\{\zeta_j\leq s\}\cap\{a_i<\zeta_{k+1}\},\ \Big(\big(\{\ell\in C\}\cap\{\zeta_j\leq t\}\big)\cup\{\zeta_j>t\}\Big)\cap\{a_i<\zeta_{k+1}\},\notag\\
&&\ \ \ \ \ s,t\leq a_i, \ C\in\mathcal{B}(E)\bigg)\notag\\
&&=\big\{\emptyset,\ \{a_i<\zeta_{k+1}\}\big\},\notag\end{eqnarray}
we have
\begin{eqnarray}&&\mathcal{G}_{a_i}\cap\{a_i<\zeta_{k+1}\}\notag\\
&&=\bigg(\mathcal{F}_{a_i}\vee\Big(\vee_{j=1}^n\big(\sigma(1_{\{\zeta_j\leq s\}},\ s\leq a_i)\vee\sigma(\ell_j1_{\{\zeta_j\leq s\}},\  s\leq a_i)\big)\Big)\bigg)\cap\{a_i<\zeta_{k+1}\}\notag\\
&&=\bigg(\big(\mathcal{F}_{a_i}\cap\{a_i<\zeta_{k+1}\}\big)\vee\Big(\vee_{j=1}^n\big(\sigma(1_{\{\zeta_j\leq s\}},\ s\leq a_i)\vee\sigma(\ell_j1_{\{\zeta_j\leq s\}},\  s\leq a_i)\big)\cap\{a_i<\zeta_{k+1}\}\Big)\bigg)\notag\\
&&=\bigg(\big(\mathcal{F}_{a_i}\cap\{a_i<\zeta_{k+1}\}\big)\vee\Big(\vee_{j=1}^k\big(\sigma(1_{\{\zeta_j\leq s\}},\ s\leq a_i)\vee\sigma(\ell_j1_{\{\zeta_j\leq s\}},\  s\leq a_i)\big)\cap\{a_i<\zeta_{k+1}\}\Big)\bigg)\notag\\
&&=\bigg(\mathcal{F}_{a_i}\vee\Big(\vee_{j=1}^k\big(\sigma(1_{\{\zeta_j\leq s\}},\ s\leq a_i)\vee\sigma(\ell_j1_{\{\zeta_j\leq s\}},\  s\leq a_i)\big)\Big)\bigg)\cap\{a_i<\zeta_{k+1}\}\notag\\
&&=\mathcal{G}_{a_i}^k\cap\{a_i<\zeta_{k+1}\},\notag
\end{eqnarray}
which proves the existence result in \leref{lemma 2-4}. Now suppose $(\bar A_i\in\mathcal{G}_{a_i}^k)_{1\leq i\leq\infty}$ are the sets such that \eqref{2-4} holds. Define $\tilde A_1=\bar A_1,\ \tilde A_{\infty}=\emptyset,$
and
$$\tilde A_{m+1}=\bar A_{m+1}\setminus\bigcup_{j=1}^m\bar A_j,\ \ m=1,2,\dotso$$
Since for $i\neq j,\ \big(\bar A_{i}\cap\{a_{i}<\zeta_{k+1}\}\big)\cap\big(\bar A_{j}\cap\{a_{j}<\zeta_{k+1}\}\big)=\emptyset$, we have for $m=1,2,\dotso$,
\begin{eqnarray}
\bar A_{m+1}\cap\{a_{m+1}<\zeta_{k+1}\}&\supset&\tilde A_{m+1}\cap\{a_{m+1}<\zeta_{k+1}\}\notag\\
&=&\big(\bar A_{m+1}\cap\{a_{m+1}<\zeta_{k+1}\}\big)\setminus\bigcup_{j=1}^m\big(\bar A_j\cap\{a_{m+1}<\zeta_{k+1}\}\big)\notag\\
&\supset&\big(\bar A_{m+1}\cap\{a_{m+1}<\zeta_{k+1}\}\big)\setminus\bigcup_{j=1}^m\big(\bar A_j\cap\{a_j<\zeta_{k+1}\}\big)\notag\\
&=&\big(\bar A_{m+1}\cap\{a_{m+1}<\zeta_{k+1}\}\big).\notag
\end{eqnarray}
Therefore, $\tilde A_{m+1}\cap\{a_{m+1}<\zeta_{k+1}\}=\bar A_{m+1}\cap\{a_{m+1}<\zeta_{k+1}\}$, and thus $(\tilde A_i\in\mathcal{G}_{a_i}^k)_{1\leq i\leq\infty}$ are the disjoint sets such that \eqref{2-4} holds. This completes the proof of Lemma~\ref{lemma 2-4}.
\hfill $\square$

Now let us continue with the proof of \thref{theorem 3-1}. From \leref{lemma 2-4}, we have
$$\{\tau<\zeta_{k+1}\}=\bigcup_{1\leq i\leq\infty}\big(\{a_i<\zeta_{k+1}\}\cap\tilde A_i\big),$$
where $(\tilde A_i\in\mathcal{G}_{a_i}^k)_{1\leq i\leq\infty}$ are disjoint sets such that \eqref{2-4} holds. Define $\mathbb{G}^k$-stopping time
$$\tau^k=\sum_{1\leq i\leq\infty}a_i1_{\tilde A_i}.$$
Since
$$\tilde A_i\cap\{\tau<\zeta_{k+1}\}= \tilde A_i\cap\bigcup_{1\leq j\leq\infty}\big(\{a_j<\zeta_{k+1}\}\cap\tilde A_j\big)=\{a_i<\zeta_{k+1}\}\cap\tilde A_i=\{\tau<\zeta_{k+1}\}\cap A_i,$$
we have 
$$\tau^k 1_{\{\tau<\zeta_{k+1}\}}=\sum_{1\leq i\leq\infty}a_i 1_{\tilde A_i\cap\{\tau<\zeta_{k+1}\}}=\sum_{1\leq i\leq\infty}a_i 1_{A_i\cap\{\tau<\zeta_{k+1}\}}=\tau 1_{\{\tau<\zeta_{k+1}\}}.$$
Also,
$$\{\tau<\zeta_{k+1}\}=\bigcup_{1\leq i\leq\infty}\Big(\{a_i<\zeta_{k+1}\}\cap A_i\Big)=\bigcup_{1\leq i\leq\infty}\Big(\{a_i<\zeta_{k+1}\}\cap\tilde A_i\Big)=\{\tau^k<\zeta_{k+1}\}.$$
\medskip

\noindent\textbf{Step 2}: We will show that for any $\mathbb{G}$-stopping time $\tau$, there exists a $\mathbb{G}^k$-stopping time $\tau^k$, such that \eqref{2-5} holds.
Define the $\mathbb{G}$-stopping times
$$\tau_m:=\sum_{j=0}^\infty\frac{j+1}{2^m} \cdot 1_{\{\frac{j}{2^m}\leq\tau<\frac{j+1}{2^m}\}}+\infty\cdot 1_{\{\tau=\infty\}},\ \ m=1,2,\dotso$$
By Step 1, there exists a $\mathbb{G}^k$-stopping time $\tau_m^k$, such that
\begin{equation}\label{2-6}\tau_m^k 1_{\{\tau_m<\zeta_{k+1}\}}=\tau_m 1_{\{\tau_m<\zeta_{k+1}\}}\ \ \ \text{and}\ \ \ \{\tau_m<\zeta_{k+1}\}=\{\tau_m^k<\zeta_{k+1}\}.\end{equation}
Define $\tau^k:=\limsup_{m\rightarrow\infty}\tau_m^k$. Since $\tau_m\searrow\tau$, by taking \lq\lq $\limsup$\rq\rq \ on both side of \eqref{2-6}, we have \eqref{2-5}.
\medskip

\noindent\textbf{Step 3}: From Step 2, we know that for any $\mathbb{G}$-stopping time $\tau$, there exists $\tau^0,\ \tau^1, \dotso, \tau^{n-1}$ being $\mathbb{F}$, $\mathbb{G}^1$, $\dotso,\ \mathbb{G}^{n-1}$-stopping times respectively, such that \eqref{2-5} holds, for $k=0,\dotso,n-1$. Let $\tau^n:=\tau$, then we have
\begin{eqnarray}
\tau&=&\tau 1_{\{\tau<\zeta_1\}}+\sum_{k=1}^{n-1}\tau1_{\{\zeta_k\leq \tau <\zeta_{k+1}\}}+\tau 1_{\{\zeta_n\leq\tau\}}\notag\\
&=&\tau^0 1_{\{\tau<\zeta_1\}}+\sum_{k=1}^{n-1}\tau^k1_{\{\zeta_k\leq \tau <\zeta_{k+1}\}}+\tau^n 1_{\{\zeta_n\leq\tau\}}\notag\\
&=&\tau^0 1_{\{\tau<\zeta_1\}}+\sum_{k=1}^{n-1}\tau^k1_{\{\tau\geq\zeta_1\}\cap\dotso\cap\{\tau\geq\zeta_k\}\cap\{\tau<\zeta_{k+1}\}}+\tau^n 1_{\{\tau\geq\zeta_1\}\cap\dotso\cap\{\tau\geq\zeta_n\}}\notag\\
&=&\tau^0 1_{\{\tau^0<\zeta_1\}}+\sum_{k=1}^{n-1}\tau^k1_{\{\tau^0\geq\zeta_1\}\cap\dotso\cap\{\tau^{k-1}\geq\zeta_k\}\cap\{\tau^k<\zeta_{k+1}\}}+\tau^n 1_{\{\tau^0\geq\zeta_1\}\cap\dotso\cap\{\tau^{n-1}\geq\zeta_n\}}.\notag
\end{eqnarray}
We will modify the decomposition so that it satisfies \eqref{eq:tgtjt}. For $k=1,\dotso,n$, define $\mathbb{G}^k$-stopping time
$$
\tilde\tau^k = \left\{ \begin{array}{rl}
 \tau^k, &\mbox{ $\tau^k\geq\zeta_k,$} \\
  \zeta_k, &\mbox{ $\tau^k<\zeta_k.$}
       \end{array} \right.
$$
and let $\tilde\tau^0:=\tau^0$. Then for $k=1,\dotso,n$, $\tilde\tau^k\geq\zeta_k$, and 
$$\{\tilde\tau^k<\zeta_{k+1}\}=\{\tau^k<\zeta_{k+1}\}=\{\tau<\zeta_{k+1}\},\ \ k=0,\dotso,n-1.$$
For $k=1,\dotso,n-1$, since $\{\zeta_k\leq \tau <\zeta_{k+1}\}\subset\{\tau=\tau^k\}$, we have
\begin{eqnarray}
&&\hspace{-3cm}\{\tau^0\geq\zeta_1\}\cap\dotso\cap\{\tau^{k-1}\geq\zeta_k\}\cap\{\tau^k<\zeta_{k+1}\}\notag\\
&&=\{\zeta_k\leq \tau <\zeta_{k+1}\}=\{\zeta_k\leq \tau <\zeta_{k+1}\}\cap\{\tau=\tau^k\}\subset\{\tau^k\geq\zeta_k\}.\notag
\end{eqnarray}
Also $\{\tau\geq\zeta_n\}\subset\{\tau=\tau_n\}$ implies
$$\{\tau^0\geq\zeta_1\}\cap\dotso\cap\{\tau^{n-1}\geq\zeta_n\}=\{\tau\geq\zeta_n\}=\{\tau\geq\zeta_n\}\cap\{\tau=\tau^n\}\subset\{\tau^n\geq\zeta_n\}.$$
Therefore, we have
\begin{eqnarray}
\tau&=&\tau^0 1_{\{\tau^0<\zeta_1\}}+\sum_{k=1}^{n-1}\tau^k1_{\{\tau^0\geq\zeta_1\}\cap\dotso\cap\{\tau^{k-1}\geq\zeta_k\}\cap\{\tau^k<\zeta_{k+1}\}}+\tau^n 1_{\{\tau^0\geq\zeta_1\}\cap\dotso\cap\{\tau^{n-1}\geq\zeta_n\}}\notag\\
&=&\tilde\tau^0 1_{\{\tau^0<\zeta_1\}}+\sum_{k=1}^{n-1}\tilde\tau^k1_{\{\tau^0\geq\zeta_1\}\cap\dotso\cap\{\tau^{k-1}\geq\zeta_k\}\cap\{\tau^k<\zeta_{k+1}\}}+\tilde\tau^n 1_{\{\tau^0\geq\zeta_1\}\cap\dotso\cap\{\tau^{n-1}\geq\zeta_n\}}\notag\\
&=&\tilde\tau^0 1_{\{\tilde\tau^0<\zeta_1\}}+\sum_{k=1}^{n-1}\tilde\tau^k1_{\{\tilde\tau^0\geq\zeta_1\}\cap\dotso\cap\{\tilde\tau^{k-1}\geq\zeta_k\}\cap\{\tilde\tau^k<\zeta_{k+1}\}}+\tilde\tau^n 1_{\{\tilde\tau^0\geq\zeta_1\}\cap\dotso\cap\{\tilde\tau^{n-1}\geq\zeta_n\}}.\notag
\end{eqnarray}
This ends the proof of Theorem~\ref{theorem 3-1}.
\end{proof} 
 In the rest of the paper, we will use the notation $\tau\sim(\tau^0,\dotso,\tau^n)$ for the $\mathbb{G}$-stopping time $\tau$ if it has the decomposition from \eqref{2-3}. The next result shows that the decomposition of $\tau$ in \eqref{2-3} is unique, in the sense that the terms in the sum of $\tau$'s representation are the same even for different $(\tau^0,\dotso,\tau^n)$'s in the representation. (Note that one can modify the stopping times $\tau^i$ after the jump times $\zeta_{i+1}$.)
\begin{proposition}\label{prop 3-3}
 Let $\tau\sim(\tau^0,\dotso,\tau^n)$ be a $\mathbb{G}$-stopping time. Then $\{\tau^0<\zeta_1\}=\{\tau<\zeta_1\}$, $\{\tau^0\geq\zeta_1\}\cap\dotso\cap\{\tau^{k-1}\geq\zeta_k\}\cap\{\tau^k<\zeta_{k+1}\}=\{\zeta_k\leq\tau<\zeta_{k+1}\}$ for $k=1,\dotso,n-1$, and $\{\tau^0\geq\zeta_1\}\cap\dotso\cap\{\tau^{n-1}\geq\zeta_n\}=\{\zeta_n\leq\tau\}$. Therefore,
\begin{equation}\label{2-7}\tau=\tau^0 1_{\{\tau<\zeta_1\}}+\sum_{k=1}^{n-1}\tau^k1_{\{\zeta_k\leq \tau <\zeta_{k+1}\}}+\tau^n 1_{\{\zeta_n\leq\tau\}}.\notag\end{equation}
\end{proposition}
\begin{proof}
Let $A_0:=\{\tau^0<\zeta_1\}$, $A_n:=\{\tau^0\geq\zeta_1\}\cap\dotso\cap\{\tau^{n-1}\geq\zeta_n\}$, and
$$A_k:=\{\tau^0\geq\zeta_1\}\cap\dotso\cap\{\tau^{k-1}\geq\zeta_k\}\cap\{\tau^k<\zeta_{k+1}\},\ \ k=1,\dotso,n-1.$$
Let $B_0:=\{\tau<\zeta_1\}$, $B_n:=\{\zeta_n\leq\tau\}$, and $B_k:=\{\zeta_k\leq\tau<\zeta_{k+1}\}, \ k=1,\dotso,n-1$. In the set $A_i$, we have $\tau=\tau^i$, which implies $\zeta_i\leq\tau<\zeta_{i+1}$, and thus $A_i\subset B_i$, for $i=1,\dotso,n-1$. Similarly, $A_0\subset B_0$ and $A_n\subset B_n$. Since $(A_i)_{i=0}^n$ and $(B_i)_{i=0}^n$ are mutually disjoint respectively, and $\Omega=\bigcup_{i=0}^nA_i=\bigcup_{i=0}^nB_i$, we have $A_i=B_i,\ i=0,\dotso,n$.
\end{proof}

The last proposition generalizes the decomposition result given in Theorem (A2.3) of \cite{MR1283589} on page 261 (also see Theorem T33 of \cite{MR636252} on page 308) from the stopping times of piecewise deterministic Markov processes to the stopping times of jump diffusions.

\begin{proposition}\label{prop 2-6}
Let $T>0$ be a constant. $\tau$ is an $\mathbb{G}$-stopping time satisfying $\tau\leq T$ if and only if $\tau$ has the decomposition \eqref{2-3}, with $\tau^0\leq T$ and $\{\zeta_k\leq T\}=\{\tau^k\leq T\},\ k=1,\dotso,n$.
\end{proposition}
\begin{proof}
If $\tau$ has the decomposition, then on the set $\{\tau^0\geq\zeta_1\}\cap\dotso\cap\{\tau^{k-1}\geq\zeta_k\}$, we have
$$T\geq\tau^0\geq\zeta_1\ \Rightarrow\  T\geq\tau^1\ \Rightarrow\  T\geq\zeta_2\ \Rightarrow\dotso\Rightarrow\  T\geq\tau^{k-1}\ \Rightarrow\  T\geq\zeta_k\ \Rightarrow T\geq\tau^k,$$
For $k=1,\dotso,n$. Thus $\tau\leq T$. 

Conversely, let $\tau\sim(\tau^0,\dotso,\tau^n)$ be a $\mathbb{G}$-stopping time satisfying $\tau\leq T$. Let $\tilde\tau^0:=\tau^0$, and 
$$
\tilde\tau^k := \left\{ \begin{array}{rl}
 \tau^k\wedge T, &\mbox{ $\zeta_k\leq T,$} \\
  \tau^k,\ \ \  &\mbox{ $\zeta_k>T.$}
       \end{array} \right.
$$
for $k=0,\dotso,n$. It can be shown that $\tilde\tau^k$ is a $\mathbb{G}^k$-stopping time. Then for $k=1,\dotso,n-1$,
$$\zeta_k\leq\tau<\zeta_{k+1}\ \Rightarrow\ \tau^k=\tau\leq T\ \Rightarrow\ \tilde\tau^k=\tau^k.$$
Similarly, $\zeta_n\leq\tau\ \Rightarrow\ \tilde\tau^n=\tau^n$. Therefore, 
$$\tau=\tilde\tau^0 1_{\{\tau<\zeta_1\}}+\sum_{k=1}^{n-1}\tilde\tau^k1_{\{\zeta_k\leq \tau <\zeta_{k+1}\}}+\tilde\tau^n 1_{\{\zeta_n\leq\tau\}}.$$
Easy to see $\tilde\tau^k\geq\zeta_k$ and $\{\zeta_k\leq T\}=\{\tilde\tau^k\leq T\},\ k=1,\dotso,n$. It remains to show $A_i=B_i, i=0,\dotso,n$, where $A_0:=\{\tau^0<\zeta_1\}$, $A_n:=\{\tau^0\geq\zeta_1\}\cap\dotso\cap\{\tau^{n-1}\geq\zeta_n\}$, 
$$A_k:=\{\tau^0\geq\zeta_1\}\cap\dotso\cap\{\tau^{k-1}\geq\zeta_k\}\cap\{\tau^k<\zeta_{k+1}\},\ \ k=0,\dotso,n-1,$$
and $B_0:=\{\tilde\tau^0<\zeta_1\}$, $B_n:=\{\tilde\tau^0\geq\zeta_1\}\cap\dotso\cap\{\tilde\tau^{n-1}\geq\zeta_n\}$, 
$$B_k:=\{\tilde\tau^0\geq\zeta_1\}\cap\dotso\cap\{\tilde\tau^{k-1}\geq\zeta_k\}\cap\{\tilde\tau^k<\zeta_{k+1}\},\ \ k=0,\dotso,n-1.$$
Easy to see $A_0=B_0$ and $A_n\supset B_n$. Now for $k=1,\dotso,n-1$,
\begin{eqnarray}
&&\{\tau^0\geq\zeta_1\}\cap\dotso\cap\{\tau^{k-1}\geq\zeta_k\}\cap\{\tilde\tau^k<\zeta_{k+1}\}\notag\\
&&\subset\{\tau^0\geq\zeta_1\}\cap\dotso\cap\{\tau^{k-1}\geq\zeta_k\}\cap\Big(\{\tau^k<\zeta_{k+1}\}\cup\{T<\zeta_{k+1}\}\Big).\notag
\end{eqnarray}
Since
$$\{\tau^0\geq\zeta_1\}\cap\dotso\cap\{\tau^{k-1}\geq\zeta_k\}\cap\{T<\zeta_{k+1}\}\cap\{\tau^k\geq\zeta_{k+1}\}=\emptyset,$$
we have
$$\{\tau^0\geq\zeta_1\}\cap\dotso\cap\{\tau^{k-1}\geq\zeta_k\}\cap\{T<\zeta_{k+1}\}\subset\{\tau^0\geq\zeta_1\}\cap\dotso\cap\{\tau^{k-1}\geq\zeta_k\}\cap\{\tau^k<\zeta_{k+1}\}.$$
Hence, for $k=1,\dotso,n-1$,
\begin{eqnarray}
B_k&\subset&\{\tau^0\geq\zeta_1\}\cap\dotso\cap\{\tau^{k-1}\geq\zeta_k\}\cap\{\tilde\tau^k<\zeta_{k+1}\}\notag\\
&=&\{\tau^0\geq\zeta_1\}\cap\dotso\cap\{\tau^{k-1}\geq\zeta_k\}\cap\{\tau^k<\zeta_{k+1}\}=A_k\notag
\end{eqnarray}
Since $\bigcup_{k=0}^nA_k=\bigcup_{k=0}^nB_k=\Omega$, and $(A_k)_{k=0}^n$ and $(B_k)_{k=0}^n$ are mutually disjoint respectively, we have $A_k=B_k,\ k=0,\dotso,n$.
\end{proof}

\section{Decomposition of expectations of $\mathbb{G}$-optional processes}\label{sec:expectation}

The main result in this section is \thref{theorem 4-3}, which shows that the expectation of  a stopped $\mathbb{G}$-optional process can be calculated using a backward induction, where each step is an expectation with respect to the Brownian filtration.

\noindent\textbf{Standing Assumption:} 
For the rest of the paper, we assume there exists a conditional probability density function $\alpha\in\mathcal{O}_{\mathbb{F}}(\Delta_n, E^n)$, such that
\begin{equation}\label{3-6}
\begin{split}
\mathbb{P}\big[(\zeta_1,\dotso,\zeta_n,\ell_1,\dotso,\ell_n)&\in d\theta_1\dotso d\theta_n de_1\dotso de_n|\mathcal{F}_t]
\\&=\alpha_t(\theta_1,\dotso,\theta_n,e_1,\dotso,e_n)d\theta_1\dotso d\theta_n \eta(de_1)\dotso \eta(de_n), \ \ \text{a.s.},
\end{split}
\end{equation}
where $d\theta_k$ is the Lebesgue measure, and $\eta(de_k)$ is some probability measure which may depend on $(\pmb\theta_{k-1},\pmb e_{k-1})$ (e.g., transition kernel), for $k=1,\dotso, n$. We also assume that the map $t\rightarrow\alpha_t$ is right continuous and
\begin{equation}\label{u.i.}
\mathbb{E}\left[\int_{E^n}\int_{\Delta_n}\sup_{t\geq 0}\alpha_t(\pmb\theta_n,\pmb e_n)d\theta_1\dotso d\theta_n \eta(de_1)\dotso \eta(de_n)\right]<\infty.
\end{equation}

Following \cite{pham1}, let us set $\alpha_t^n(\pmb\theta_n,\pmb e_n)=\alpha_t(\pmb\theta_n,\pmb e_n)$, and
\begin{equation}\label{eq:alphak}
\alpha_t^k(\pmb\theta_k,\pmb e_k)=\int_E\int_t^\infty \alpha_t^{k+1}(\pmb\theta_k,\theta_{k+1},\pmb e_k,e_{k+1})\ d\theta_{k+1}\eta(de_{k+1}),\ \ \ k=0,\dotso,n-1.
\end{equation}
Note that $\alpha=0$ when $\theta_1,\dotso,\theta_n$ are not in an ascending order. As a result, for $k=0,\dotso,n-1$,
$$\alpha_t^k(\pmb\theta_k,\pmb e_k)=\int_{E^k}\int_t^\infty\int_{\theta_{k+1}}^\infty\dotso\int_{\theta_{n-1}}^\infty\alpha_t(\pmb\theta_n,\pmb e_n)\ d\theta_n\dotso d\theta_{k+1}\eta(de_n)\dotso \eta(de_{k+1}).$$
Hence $ \mathbb{P}[\zeta_1>t|\mathcal{F}_t]=\alpha_t^0$, and for $k=1,\dotso,n-1$,
$$\mathbb{P}[\zeta_{k+1}>t|\mathcal{F}_t]=\int_{E^k}\int_{\Delta_k}\alpha_t^k(\theta_1,\dotso,\theta_k,e_1,\dotso,e_k)\ d\theta_1\dotso d\theta_k \eta(de_1)\dotso \eta(de_k).$$
Therefore, $\alpha^k$ can be interpreted as the survival density of $\zeta_{k+1}$. 

Let us recall the following lemma from \cite{pham1}.
\begin{lemma}\label{lemma 2-1}
 Any process $Z=(Z_t)_{t\geq 0}$ is $\mathbb{G}$-optional if and only if it has the decomposition:
\begin{equation}\label{2-2} Z_t=Z_t^0 1_{\{t<\zeta_1\}}+\sum_{k=1}^{n-1}Z_t^k(\pmb\zeta_k,\pmb\ell_k)1_{\{\zeta_k\leq t<\zeta_{k+1}\}}+Z_t^n(\pmb\zeta_n,\pmb\ell_n)1_{\{\zeta_n\leq t\}},\end{equation}
for some $Z^k\in\mathcal{O}_{\mathbb{F}}(\Delta_k,E^k)$, for $k=0,\dotso,n$. A similar decomposition result holds for any $\mathbb{G}$-predictable process.
\end{lemma}

We will use the notation $Z\sim(Z^0,\dotso,Z^n)$ for the $\mathbb{G}$-optional (resp. predictable) process  $Z$ from the decomposition \eqref{2-2}. Let $Z\sim(Z^0,\dotso,Z^n)$ be a $\mathbb{G}$-optional process, and $\tau\sim(\tau^0,\dotso,,\tau^n)$  be a $\mathbb{G}$-stopping time. Then from \leref{lemma 2-1} and \prref{prop 3-3}, $Z_{\tau}$ has the decomposition:
\begin{equation}\label{3-1} Z_\tau=Z_{\tau^0}^0 1_{\{\tau<\zeta_1\}}+\sum_{k=1}^{n-1}Z_{\tau^k}^k 1_{\{\zeta_k\leq \tau< \zeta_{k+1}\}}+Z_{\tau^n}^n 1_{\{\zeta_n\leq\tau\}}.\end{equation} 

The following lemma will be used for the rest of the paper:
\begin{lemma}
$\tau^k(\pmb\zeta_k,\pmb\ell_k)$ is a $\mathbb{G}^k$-stopping time satisfying $\tau^k\geq\zeta_k$ if and only if for any fixed $(\pmb\theta_k,\pmb e_k)\in\Delta_k\times E^k$, $\tau^k(\pmb\theta_k,\pmb e_k)$ is an $\mathbb{F}$-stopping time satisfying $\tau^k(\pmb\theta_k,\pmb e_k)\geq\theta_k$ and $\tau^k(\pmb\theta_k,\pmb e_k)$ is measurable with respect to $(\pmb\theta_k,\pmb e_k)$.
\end{lemma}
\begin{proof}
If $\tau^k(\pmb\theta_k,\pmb e_k)$ is an $\mathbb{F}$-stopping time satisfying $\tau^k(\pmb\theta_k,\pmb e_k)\geq\theta_k$ and is measurable with respect to $(\pmb\theta_k,\pmb e_k)$, then $1_{\{\tau^k(\pmb\theta_k,\pmb e_k)\leq t\}}\cdot 1_{\{\theta_k\leq t\}}\in\mathcal{O}_\mathbb{F}(\Delta_k,E^k)$. By \leref{lemma 2-1} (here $n=k$), $1_{\{\tau^k(\pmb\zeta_k,\pmb\ell_k)\leq t\}}=1_{\{\tau^k(\pmb\zeta_k,\pmb\ell_k)\leq t\}}\cdot 1_{\{\zeta_k\leq t\}}$ is a $\mathbb{G}^k$-optional process. Then $\{\tau^k(\pmb\zeta_k,\pmb\ell_k)\leq t\}=\big\{1_{\{\tau^k(\pmb\zeta_k,\pmb\ell_k)\leq t\}}=1\big\}\in\mathcal{G}_t^k$. Hence, $\tau^k(\pmb\zeta_k,\pmb\ell_k)$ is a $\mathbb{G}^k$-stopping time. Conversely, if $\tau^k(\pmb\zeta_k,\pmb\ell_k)$ is a $\mathbb{G}^k$-stopping time, then the $\mathbb{G}^k$-optional process $1_{\{\tau^k(\pmb\zeta_k,\pmb\ell_k)\leq t\}}$ has the representation from \leref{lemma 2-1}. Thus, for fixed $(\pmb\theta_k,\pmb e_k)$, $1_{\{\tau^k(\pmb\theta_k,\pmb e_k)\leq t\}}$ is $\mathbb{F}$-optional, which implies that $\tau^k(\pmb\theta_k,\pmb e_k)$ is an $\mathbb{F}$-stopping time.
\end{proof}

The following theorem is the main result of this section.
\begin{theorem}\label{theorem 4-3}
Let $Z\sim(Z^0,,\dotso,Z^n)$ be a nonnegative (or bounded), right continuous $\mathbb{G}$-optional process, and $\tau\sim(\tau^0,\dotso,\tau^n)$ be a finite $\mathbb{G}$-stopping time satisfying $\tau\leq T$, where $T\in[0,\infty]$ is a constant. The expectation $\displaystyle\mathbb{E}\big[Z_\tau\big]$ can be computed by a backward induction as
\begin{equation}\label{3-2}\mathbb{E}\big[Z_\tau\big]=J_0,\notag\end{equation}
where $J_0,\dotso,J_n$ are given by
\begin{eqnarray}\label{3-3}&&J_n(\pmb\theta_n,\pmb e_n)=\mathbb{E}\Big[Z_{\tau^n}^n\alpha_{\tau^n}^n(\pmb\theta_n,\pmb e_n)\big|\mathcal{F}_{\theta_n}\Big],\ \ \ (\pmb\theta_n,\pmb e_n)\in\Delta_n(T)\times E^n,\\
\label{3-4} &&J_k(\pmb\theta_k,\pmb e_k)=\mathbb{E}\bigg[Z_{\tau^k}^k\alpha_{\tau^k}^k(\pmb\theta_k,\pmb e_k)+\int_{\theta_k}^{\tau^k(\pmb\theta_k,\pmb e_k)\wedge T}\int_EJ_{k+1}(\pmb\theta_k,\theta_{k+1},\pmb e_k,e_{k+1})\eta(de_{k+1})d\theta_{k+1}\Big|\mathcal{F}_{\theta_k}\bigg],\hspace{-0.3cm}
\end{eqnarray}
$(\pmb\theta_k,\pmb e_k)\in\Delta_k(T)\times E^k$, for $k=0,\dotso,n-1$.
\end{theorem}
\begin{proof}
For the sake of simplicity, let us assume $n=2$. Using \eqref{3-3} and \eqref{3-4},  plugging $J_2$ into $J_1$, and then $J_1$ into $J_0$, we obtain
\begin{eqnarray}J_0&=&\mathbb{E}\Big[Z_{\tau^0}^0\alpha_{\tau_0}^0\Big]+\mathbb{E}\bigg[\int_0^{\tau^0\wedge T}\int_E\mathbb{E}\Big[Z_{\tau^1(\theta_1,e_1)}^1\cdot\alpha_{\tau^1(\theta_1, e_1)}^1\big|\mathcal{F}_{\theta_1}\Big]\eta(de_1)d\theta_1\bigg]\notag\\
&+&\mathbb{E}\Bigg[\int_0^{\tau_0\wedge T}\int_E\mathbb{E}\bigg[\int_0^{\tau_1(\theta_1,e_1)\wedge T}\int_E\mathbb{E}\Big[Z_{\tau^2(\theta_1,\theta_2,e_1,e_2)}^2\cdot\alpha_{\tau^2}^2\big|\mathcal{F}_{\theta_2}\Big]\eta(de_2)d\theta_2\Big|\mathcal{F}_{\theta_1}\bigg]\eta(de_1)d\theta_1\Bigg].\notag
\end{eqnarray}
On the right side of the equation above, let us denote the fist term by I, the second term by II, and the third term by III. We can show that
\begin{eqnarray}
\text{I}&=&\mathbb{E}\Big[\int_{E^2}\int_{\Delta_2}Z_{\tau^0}^0\cdot 1_{\{\theta_1>\tau^0\}}\cdot\alpha_{\tau^0}(\theta_1,\theta_2,e_1,e_2)\ d\theta_1 d\theta_2 \eta(de_1)\eta(de_2)\Big],\notag\\
\text{II}&=&\mathbb{E}\bigg[\int_{E^2}\int_{\Delta_2}Z_{\tau^1(\theta_1,e_1)}^1\cdot 1_{\{\theta_1\leq T\}}\cdot 1_{\{\tau^0\geq\theta_1\}\cap\{\tau^1(\theta_1,e_1)<\theta_2\}}\cdot\alpha_{\tau^1}\ d\theta_1 d\theta_2 \eta(de_1) \eta(de_2)\bigg],\notag\\
\text{III}&=&\mathbb{E}\bigg[\int_{E^2}\int_{\Delta_2}Z_{\tau 2(\theta_1,\theta_2,e_1,e_2)}^2\cdot 1_{\{\theta_1,\theta_2\leq T\}}\cdot 1_{\{\tau^0\geq\theta_1\}\cap\{\tau^1(\theta_1,e_1)\geq\theta_2\}}\cdot\alpha_{\tau^2}\ d\theta_1 d\theta_2 \eta(de_1) \eta(de_2)\bigg].\notag
\end{eqnarray}
For fixed $(\theta_1,\theta_2,e_1,e_2)\in\Delta_2\times E^2$, from \prref{prop 3-3}, we have $\{\tau^0\geq\theta_1\}\cap\{\tau^1<\theta_2\}=\{\theta_1\leq\tau<\theta_2\}\subset\{\theta_1\leq T\}$, and 
$\{\tau^0\geq\theta_1\}\cap\{\tau^1\geq\theta_2\}=\{\theta_2\leq\tau\}\subset\{\theta_1,\theta_2\leq T\}$. Hence,
\begin{eqnarray}
&&Z_\tau(\theta_1,\theta_2,e_1,e_2)=Z_{\tau^0}^0\cdot 1_{\{\tau^0<\theta_1\}}+Z_{\tau^1}^1\cdot 1_{\{\tau^0\geq\theta_1\}}\cdot 1_{\{\tau^1<\theta_2\}}+Z_{\tau^2}^2 \cdot 1_{\{\tau^0\geq\theta_1\}}\cdot 1_{\{\tau^1\geq\theta_2\}}\notag\\
&&\hspace{1.5cm}=Z_{\tau^0}^0\cdot 1_{\{\tau^0<\theta_1\}}+Z_{\tau^1}^1\cdot 1_{\{\theta_1\leq T\}}\cdot 1_{\{\tau^0\geq\theta_1\}\cap\{\tau^1<\theta_2\}}+Z_{\tau^2}^2 \cdot 1_{\{\theta_1,\theta_2\leq T\}}\cdot 1_{\{\tau^0\geq\theta_1\}\cap\{\tau^1\geq\theta_2\}}.\notag
\end{eqnarray}
Therefore, we have
$$J_0=\text{I}+\text{II}+\text{III}=\mathbb{E}\bigg[\int_{E^2}\int_{\Delta_2} Z_\tau(\theta_1,\theta_2,e_1,e_2)\cdot\alpha_\tau(\theta_1,\theta_2,e_1,e_2)\ d\theta_1d\theta_2\eta(de_1)\eta(de_2)\bigg].$$
We will show in two steps that $J_0=\mathbb{E}[Z_\tau]$.\\
\textbf{Step 1}: If $\displaystyle\tau=\sum_{k=0}^\infty a_k 1_{A_k}$, where $0\leq a_0<a_1\dotso<\infty$, and $A_k\in\mathcal{G}_{a_k},\ k=0,1,\dotso$, then
\begin{eqnarray}&&\mathbb{E}[Z_{\tau}]=\sum_{k=0}^\infty\mathbb{E}\big[Z_{a_k}1_{A_k}\big]\notag\\
&&=\sum_{k=0}^\infty\mathbb{E}\left[\int_{E^2}\int_{\Delta_2}Z_{a_k}(\theta_1,\theta_2,e_1,e_2)1_{A_k}(\theta_1,\theta_2,e_1,e_2)\alpha_{a_k}(\theta_1,\theta_2,e_1,e_2)d\theta_1d\theta_2\eta(de_1)\eta(de_2)\right]\notag\\
&&=\mathbb{E}\left[\int_{E^2}\int_{\Delta_2}\left(\sum_{k=0}^\infty Z_{a_k}(\theta_1,\theta_2,e_1,e_2)1_{A_k}(\theta_1,\theta_2,e_1,e_2)\alpha_{a_k}(\theta_1,\theta_2,e_1,e_2)\right)d\theta_1d\theta_2\eta(de_1)\eta(de_2)\right]\notag\\
&&=\mathbb{E}\bigg[\int_{E^2}\int_{\Delta^2} Z_\tau(\theta_1,\theta_2,e_1,e_2)\cdot\alpha_{\tau}(\theta_1,\theta_2,e_1,e_2)\ d\theta_1d\theta_2\eta(de_1)\eta(de_2)\bigg],\notag
\end{eqnarray}
where the second equality above follows from Proposition 2.1 in \cite{pham1}.\\
\textbf{Step 2}: In general, let $\tau$ be any finite $\mathbb{G}$-stopping time. Define
$$\tau_m:=\sum_{j=0}^\infty\frac{j+1}{2^m} \cdot 1_{\{\frac{j}{2^m}\leq\tau<\frac{j+1}{2^m}\}},\ \ m=1,2,\dotso$$
For fixed $N\in(0,\infty)$, Step 1 implies that
$$\mathbb{E}\big[Z_{\tau^m}\wedge N\big]=\mathbb{E}\bigg[\int_{E^2}\int_{\Delta^2} \big(Z_{\tau^m}\big(\theta_1,\theta_2,e_1,e_2)\wedge N\big)\cdot\alpha_{\tau^m}(\theta_1,\theta_2,e_1,e_2)\ d\theta_1d\theta_2\eta(de_1)\eta(de_2)\bigg].$$
Thanks to \eqref{u.i.} and the right continuity of $Z_t$ and $\alpha_t$, by sending $m\rightarrow\infty$, we get
$$\mathbb{E}\big[Z_{\tau}\wedge N\big]=\mathbb{E}\bigg[\int_{E^2}\int_{\Delta^2} \big(Z_{\tau}\big(\theta_1,\theta_2,e_1,e_2)\wedge N\big)\cdot\alpha_{\tau}(\theta_1,\theta_2,e_1,e_2)\ d\theta_1d\theta_2\eta(de_1)\eta(de_2)\bigg].$$
Then letting $N\rightarrow\infty$, the result follows.
\end{proof}
\begin{remark}
When the Brownian motion and the jumps are independent, \eqref{u.i.} and the right continuity of $\alpha_t$ in the Standing Assumption trivially holds. In this case, \thref{theorem 4-3} still holds if the assumption of the right continuity of $Z_t$ is removed. In fact, it follows directly from the expectation under the product probability measure that 
$$J_0=\mathbb{E}\bigg[\int_{E^2}\int_{\Delta_2} Z_\tau(\theta_1,\theta_2,e_1,e_2)\cdot\alpha(\theta_1,\theta_2,e_1,e_2)\ d\theta_1d\theta_2\eta(de_1)\eta(de_2)\bigg]=\mathbb{E}[Z_{\tau}].$$
The same applies for \thref{theorem 4-2} and \prref{prop 4-3}.
\end{remark}

\section{Decomposition of $\mathbb{G}$-controller-stopper problems}\label{sec:control}

\thref{theorem 4-2} and \prref{prop 4-3} are the main results for this section, which  decompose the global $\mathbb{G}$-controller-stopper problems into a backward induction, where each step is a controller-stopper problem with respect to the Brownian filtration. 

A control is a $\mathbb{G}$-predictable process $\pi\sim(\pi^ 0,\dotso,\pi^n)$, where $\pi^k\in\mathcal{P}_{\mathbb{F}}(\Delta_k,E^k)$ is valued in a set $A^k$ in some Euclidian space, for $k=0,\dotso,n$. We denote by $\mathcal{P}_{\mathbb{F}}(\Delta_k,E^k;A^k)$ the set of elements in $\mathcal{P}_{\mathbb{F}}(\Delta_k,E^k)$ valued in $A^k,\ k=0,\dotso,n$. We require that all the $\mathbb{G}$-stopping times we consider here are valued in $[0,T]$, where $T\in(0,\infty]$ is a given constant.  A trading strategy is a pair of a control and a $\mathbb{G}$-stopping time. We will use the notation $(\pi,\tau)\sim(\pi^k,\tau^k)_{k=0}^n$ for the trading strategy if $\pi\sim(\pi^0,\dotso,\pi^n)$ and $\tau\sim(\tau^0,\dotso,\tau^n)$. A trading strategy $(\pi,\tau)\sim(\pi^k,\tau^k)_{k=0}^n$ is admissible, if for $k=0,\dotso,n, \ (\pi^k,\tau^k)\in\mathcal{A}^k\times\mathcal{T}^k$, where $\mathcal{A}^k$ is some seperable metric space of $\mathcal{P}(\Delta_k,E^k;A^k)$, and $\mathcal{T}^k$ is some set of finite $\mathbb{G}^k$-stopping times. By \prref{prop 2-6}, we let $\mathcal{T}^k$ be such that for any $\tau^k\in\mathcal{T}^k$, $\tau^k(\pmb\theta_k,\pmb e_k)\leq T$ whenever $\theta_k\leq T$. Note that $\mathcal{A}^k$ and $\mathcal{T}^k$ may depend on each other in general. We denote the set of admissible trading strategies by $\mathcal{A}_\mathbb{G}\times\mathcal{T}_\mathbb{G}$. 

The following lemma will be used for the measurable selection issue later on.
\begin{lemma}\label{lemma 4-1}
For $k=0,\dotso,n$, define the metric on $\mathcal{T}^k$ in the following way: 
$$\rho(\tau_1^k,\tau_2^k):=\mathbb{E}\Big[\int_0^{\infty}e^{-t}\big|1_{\{\tau_1^k\leq t\}}-1_{\{\tau_2^k\leq t\}}\big|dt\Big], \ \ \ \ \ \tau_1^k,\tau_2^k\in\mathcal{T}^k.$$
Then $\mathcal{T}^k$ is a separable metric space.
\end{lemma}
\begin{proof}
Since for any $\mathbb{G}^k$-stopping time $\tau^k$, $e^{-t}1_{\{\tau^k\leq t\}}$ is a $\mathbb{G}^k$-adapted process in $L^1([0,\infty)\times\Omega)$, the conclusion follows from the separability of $L^1$, see \cite{Touzi1}.
\end{proof}
Following \cite{pham1}, we describe  the formulation of a stopped controlled state process as follows:
\begin{itemize}
\item Controlled state process between jumps:
$$(x,\pi^k)\in\mathbb{R}^d\times\mathcal{A}^k\longmapsto X^{k,x,\pi^k}\in\mathcal{O}_{\mathbb{F}}(\Delta_k,E^k),\ \ k=0,\dotso,n,$$
such that
$$X_0^{0,x,\pi^0}=x,\ \ \ \ \ X_{\theta_k}^{k,\beta,\pi^k}(\pmb\theta_k,\pmb e_k)=\beta,\ \ \forall \beta\ \mathcal{F}_{\theta_k}\text{-measurable}.$$
\item Jumps of controlled state process: we have a collection of maps $\Gamma^k$ on $\mathbb{R}_+\times\Omega\times\mathbb{R}^d\times A^{k-1}\times E$, for $k=1,\dotso,n$, such that
$$(t,\omega,x,a,e)\mapsto\Gamma^k(\omega,x,a,e)\ \ \ \text{is}\ \mathcal{P}_\mathbb{F}\otimes\mathcal{B}(\mathbb{R}^d)\otimes\mathcal{B}(A^{k-1})\otimes\mathcal{B}(E)\text{-measurable}$$
\item Global controlled state process: 
$$\big(x, \pi\sim(\pi^0,\dotso,\pi^n)\big)\in\mathbb{R}^d\times\mathcal{A}_{\mathbb{G}}\longmapsto X^{x,\pi}\in\mathcal{O}_{\mathbb{G}},$$
where
\begin{equation}\label{4-1} X_t^{x,\pi}=\bar X_t^0 1_{\{t<\zeta_1\}}+\sum_{k=1}^{n-1} \bar X_t^k(\pmb\zeta_k,\pmb\ell_k) 1_{\{\zeta_k\leq t<\zeta_{k+1}\}}+\bar X_t^n(\pmb\zeta_n,\pmb\ell_n) 1_{\{\zeta_n\leq t\}},\end{equation}
with $(\bar X^0,\dotso,\bar X^n)\in\mathcal{O}_{\mathbb{F}}(\Delta_0,E^0)\times\dotso\times\mathcal{O}_{\mathbb{F}}(\Delta_n,E^n)$ with initial data
\begin{eqnarray}
&&\bar X^0=X^{0,x,\pi^0},\notag\\
&&\bar X^k(\pmb\theta_k,\pmb e_k)=X^{k,\Gamma_{\theta_k}^k(\bar X_{\theta_k}^{k-1},\pi_{\theta_k}^{k-1},e_k),\pi^k}(\pmb\theta_k,\pmb e_k),\ \ k=1,\dotso,n.\notag
\end{eqnarray}
\item Stopped global controlled state process: given a trading strategy $(\pi,\tau)\sim(\pi^k,\tau^k)_{k=0}^n$ in $\mathcal{A}_{\mathbb{G}}\times\mathcal{T}_{\mathbb{G}}$, let $X^{x,\pi}$ be the process from \eqref{4-1}, then the stopped controlled state process is:
\begin{equation}\label{4-2} X_\tau^{x,\pi}=\bar X_{\tau^0}^0 1_{\{\tau<\zeta_1\}}+\sum_{k=1}^{n-1} \bar X_{\tau^k}^k(\pmb\zeta_k,\pmb\ell_k) 1_{\{\zeta_k\leq \tau<\zeta_{k+1}\}}+\bar X_{\tau^n}^n(\pmb\zeta_n,\pmb\ell_n) 1_{\{\zeta_n\leq \tau\}}.\end{equation}
\end{itemize}

Assume $U\sim(U^0,\dotso,U^n)$ is bounded (nonnegative, nonpositive), $\mathcal{O}_\mathbb{G}\otimes\mathcal{B}(\mathbb{R}^d)$-measurable which gives the terminal payoff $U_t$ at time $t$ . Consider the two types of the controller-stopper problems:
\begin{equation}\label{4-3}V^0(x)=\sup_{\tau\in\mathcal{T}_{\mathbb{G}}}\sup_{\pi\in\mathcal{A}_{\mathbb{G}}}\mathbb{E}\big[U_\tau(X_\tau^{x,\pi})\big],\ \ x\in\mathbb{R}^d,\end{equation}
\begin{equation}\label{4-4}\mathfrak{V}^0(x)=\sup_{\pi\in\mathcal{A}_{\mathbb{G}}}\inf_{\tau\in\mathcal{T}_{\mathbb{G}}}\mathbb{E}\big[U_\tau(X_\tau^{x,\pi})\big],\ \ x\in\mathbb{R}^d.\end{equation}
We require that for any $x\in\mathbb{R}^d$ and admissible control $\pi$, the map $t\rightarrow U_t(X_t^{x,\pi})$ is right continuous. 

The following theorem provides a decomposition for calculating $V^0$ in \eqref{4-3}. Its proof is similar to the proof of Theorem  4.1 in \cite{pham1}.     
\begin{theorem}\label{theorem 4-2}
Define value functions $(\bar V^k)_{k=0}^n$ as
\begin{eqnarray}
\label{4-6} && \bar V^n(x,\pmb\theta_n,\pmb e_n)=\esssup_{\tau^n\in\mathcal{T}^n}\esssup_{\pi^n\in\mathcal{A}^n}\mathbb{E}\Big[U_{\tau^n}^n(X_{\tau^n}^{n,x,\pi^n},\pmb\theta_n,\pmb e_n)\cdot \alpha_{\tau^n}^n(\pmb\theta_n,\pmb e_n)\big|\mathcal{F}_{\theta_n}\Big],\ (\pmb\theta_n,\pmb e_n)\in\Delta_n(T)\times E^n,\hspace{-0.3cm}\notag\\
\label{4-7} &&\bar V^k(x,\pmb\theta_k,\pmb e_k)=\esssup_{\tau^k\in\mathcal{T}^k}\esssup_{\pi^k\in\mathcal{A}^k}\mathbb{E}\Big[U_{\tau^k}^k(X_{\tau^k}^{k,x,\pi^k},\pmb\theta_k,\pmb e_k)\cdot \alpha_{\tau^k}^k(\pmb\theta_k,\pmb e_k)\\
&&\hspace{1cm}+\int_{\theta_k}^{\tau^k}\int_E \bar V^{k+1}\Big(\Gamma_{\theta_k}^{k+1}(X_{\theta_{k+1}}^{k,x,\pi^k},\pi_{\theta_{k+1}}^k,e_{k+1}),\pmb\theta_k,\theta_{k+1},\pmb e_k,e_{k+1}\Big)\eta(de_{k+1})d\theta_{k+1}\big|\mathcal{F}_{\theta_k}\Big],\notag
\end{eqnarray}
$(\pmb\theta_k,\pmb e_k)\in\Delta_k(T)\times E^k$, for $k=0,\dotso,n-1$. Then $V^0(x)=\bar V^0(x)$.
\end{theorem}
\begin{remark}
In Equation \eqref{4-7}, the first term $U_{\tau^k}^k(X_{\tau^k}^{k,x,\pi^k},\pmb\theta_k,\pmb e_k)\cdot\alpha_{\tau^k}^k(\pmb\theta_k,\pmb e_k)$  can be interpreted as the gain when there are no jumps between $\theta_k$ and $\tau^k$, which is measured by the survival density $\alpha_{\tau^k}^k$. The second term 
$$\int_{\theta_k}^{\tau^k}\int_E \bar V^{k+1}\Big(\Gamma_{\theta_k}^{k+1}(X_{\theta_{k+1}}^{k,x,\pi^k},\pi_{\theta_{k+1}}^k,e_{k+1}),\pmb\theta_k,\theta_{k+1},\pmb e_k,e_{k+1}\Big)\eta(de_{k+1})d\theta_{k+1}$$
can be understood as the gain when there is a jump at time $\theta_{k+1}$ between $\theta_k$ and $\tau^k$.
\end{remark}
\begin{proof}[Proof of \thref{theorem 4-2}]
For $x\in\mathbb{R}^d$, $(\pi,\tau)\sim(\pi^k,\tau^k)_{k=0}^n$ in $\mathcal{A}_{\mathbb{G}}\times\mathcal{T}_{\mathbb{G}}$, define
\begin{eqnarray}
\label{4-10} && I^n(x,\pmb\theta_n,\pmb e_n,\pi,\tau)=\mathbb{E}\Big[U_{\tau^n}^n(X_{\tau^n}^{n,x,\pi^n},\pmb\theta_n,\pmb e_n)\cdot \alpha_{\tau^n}^n(\pmb\theta_n,\pmb e_n)\big|\mathcal{F}_{\theta_n}\Big],\ \ \ (\pmb\theta_n,\pmb e_n)\in\Delta_n(T)\times E^n,\notag\\
\label{4-11} && I^k(x,\pmb\theta_k,\pmb e_k,\pi,\tau)=\mathbb{E}\bigg[U_{\tau^k}^k(X_{\tau^k}^{k,x,\pi^k},\pmb\theta_k,\pmb e_k)\cdot\alpha_{\tau^k}^k(\pmb\theta_k,\pmb e_k)\notag\\
&&\hspace{0.5cm}+\int_{\theta_k}^{\tau^k}\int_E I^{k+1}\Big(\Gamma_{\theta_{k+1}}^{k+1}(X_{\theta_{k+1}}^{k,x,\pi^k},\pi_{\theta_{k+1}}^k,e_{k+1}),\pmb\theta_k,\theta_{k+1},\pmb e_k,e_{k+1},\pi,\tau\Big)\eta(de_{k+1})d\theta_{k+1}\big|\mathcal{F}_{\theta_k}\bigg],\notag
\end{eqnarray}
$(\pmb\theta_k,\pmb e_k)\in\Delta_k(T)\times E^k$, for $k=0,\dotso,n-1$. Set $\displaystyle\bar I^k(\pmb\theta_k,\pmb e_k)=I^k(\bar X_{\theta_k}^k,\pmb\theta_k,\pmb e_k,\pi,\tau),\ \ k=0,\dotso,n$. From the decomposition \eqref{4-2}, we know that $(\bar I^k)_{k=0}^n$ satisfy the backward induction formula:
\begin{eqnarray}
&& \bar I^n(\pmb\theta_n,\pmb e_n)=\mathbb{E}\Big[U_{\tau^n}^n(\bar X_{\tau^n}^n,\pmb\theta_n,\pmb e_n)\cdot\alpha_{\tau^n}^n(\pmb\theta_n,\pmb e_n)\big|\mathcal{F}_{\theta_n}\Big],\notag\\
&& \bar I^k(\pmb\theta_k,\pmb e_k)=\mathbb{E}\bigg[U_{\tau_k}^k(\bar X_{\tau^k}^k,\pmb\theta_k,\pmb e_k)\cdot\alpha_{\tau^k}^k(\pmb\theta_k,\pmb e_k)+\int_{\theta_k}^{\tau^k}\int_E \bar I^{k+1}(\pmb\theta_k,\theta_{k+1},\pmb e_k,e_{k+1})\ \eta(de_{k+1})d\theta_{k+1}\big|\mathcal{F}_{\theta_k}\bigg].\notag
\end{eqnarray}
From \thref{theorem 4-3} we have that
\begin{equation}\label{4-12}\bar I^0=I^0=\mathbb{E}\big[U_\tau(X_\tau^{x,\pi})\big].\end{equation}
Define the value function processes
\begin{equation}\label{4-13} V^k(x,\pmb\theta_k,\pmb e_k):=\esssup_{\tau\in\mathcal{A}_{\mathbb{G}}}\esssup_{\pi\in\mathcal{T}_{\mathbb{G}}}\ I^k(x,\pmb\theta_k,\pmb e_k,\pi,\tau),\end{equation}
for $k=0,\dotso,n,\ x\in\mathbb{R}^d,\ (\pmb\theta_k,\pmb e_k)\in\Delta_k(T)\times E^k$. Observe that $V^0$ defined in \eqref{4-13} is consistent with its definition in \eqref{4-3} from \eqref{4-12}. Then it remains to show that $\bar V^k=V^k$ for $k=0,\dotso,n$. For $k=n$, since $I^n(x,\pmb\theta_n,\pmb e_n,\pi,\tau)$ in fact only depends on $(\pi^n,\tau^n)$, we immediately have $\bar V^n=V^n$. Now assume $\bar V^{k+1}=V^{k+1}$, for $0\leq k\leq n-1$. Then for any $(\pi,\tau)\sim(\pi^k,\tau^k)_{k=0}^n$ in $\mathcal{A}_{\mathbb{G}}\times\mathcal{T}_{\mathbb{G}}$,
\begin{eqnarray}
I^k(x,\pmb\theta_k,\pmb e_k,\pi,\tau)&\leq&\mathbb{E}\bigg[U_{\tau^k}^k(X_{\tau^k}^{k,x,\pi^k},\pmb\theta_k,\pmb e_k)\cdot\alpha_{\tau^k}^k(\pmb\theta_k,\pmb e_k)\notag\\
&+&\int_{\theta_k}^{\tau^k}\int_E V^{k+1}\Big(\Gamma_{\theta_{k+1}}^{k+1}(X_{\theta_{k+1}}^{k,x,\pi^k},\pi_{\theta_{k+1}}^k, e_{k+1}),\pmb\theta_k,\theta_{k+1},\pmb e_k,e_{k+1}\Big)\ \eta(de_{k+1})d\theta_{k+1}\Big|\mathcal{F}_{\theta_k}\bigg]\notag\\
&\leq&\bar V^k(x,\pmb\theta_k,\pmb e_k),\notag
\end{eqnarray}
which implies that $V^k\leq\bar V^k$. 

Conversely, given $x\in\mathbb{R}^d$ and $(\pmb\theta_k,\pmb e_k)\in\Delta_k(T)\times E^k$, let us prove $V^k(x,\pmb\theta_k.\pmb e_k)\geq\bar V^k(x,\pmb\theta_k,\pmb e_k)$. Fix $(\pi^k,\tau^k)\in\mathcal{A}^k\times\mathcal{T}^k$ and the associated controlled process $X^{k,x,\pi^k}$, from the definition of $V^{k+1}$, we have that for any $\omega\in\Omega$ and $\epsilon>0$, there exists $(\pi^{\omega,\epsilon},\tau^{\omega,\epsilon})\in\mathcal{A}_{\mathbb{G}}\times\mathcal{T}_{\mathbb{G}}$, such that it is an $\epsilon e^{-\theta_{k+1}}$-optimal trading strategy for $V^{k+1}(\cdot,\pmb\theta_k,\pmb e_k)$ at $\big(\omega,\Gamma_{\theta_{k+1}}^{k+1} (X_{\theta_{k+1}}^{k,x,\pi^k},\pi_{\theta_{k+1}}^k, e_{k+1})\big)$. By the separability of the set of admissible trading strategies from \leref{lemma 4-1}, one can use a measurable selection argument (e.g., see \cite{Wagner1}) to find $(\pi^{\epsilon}, \tau^{\epsilon})\sim(\pi^{\epsilon,k},\tau^{\epsilon,k})_{k=0}^n$ in $\mathcal{A}_{\mathbb{G}}\times\mathcal{T}_{\mathbb{G}}$, such that $\pi_t^{\epsilon}(\omega)=\pi_t^{\omega,\epsilon}(\omega),\ dt\otimes d\mathbb{P}$-a.e. and $\tau^{\epsilon}(\omega)=\tau^{\omega,\epsilon}(\omega)$, a.s., and thus
\begin{eqnarray}
&&V^{k+1}\big(\Gamma_{\theta_{k+1}}^{k+1}(X_{\theta_{k+1}}^{k,x,\pi},\pi_{\theta_{k+1}}^k,e_{k+1}),\pmb\theta_k,\theta_{k+1},\pmb e_k,e_{k+1}\big)-\epsilon e^{-\theta_{k+1}}\notag\\
&&\hspace{4cm}\leq I^{k+1}\big(\Gamma_{\theta_{k+1}}^{k+1}(X_{\theta_{k+1}}^{k,x,\pi},\pi_{\theta_{k+1}}^k,e_{k+1}),\pmb\theta_k,\theta_{k+1},\pmb e_k,e_{k+1},\pi^{\epsilon},\tau^{\epsilon}\big),\ \ \ \text{a.s.}\notag
\end{eqnarray}
Consider the admissible trading strategy $(\tilde\pi^{\epsilon},\tilde\tau^{\epsilon})$ with the decomposition
$$\tilde\pi^{\epsilon}\sim(\pi^{\epsilon,0},\dotso,\pi^{\epsilon,k-1},\pi^k,\pi^{\epsilon,k+1},\dotso,\pi^{\epsilon,n})\ \ \text{and}\ \  \tilde\tau^{\epsilon}\sim(\tau^{\epsilon,0},\dotso,\tau^{\epsilon,k-1},\tau^k,\tau^{\epsilon,k+1},\dotso,\tau^{\epsilon,n}).$$
Since $I^{k+1}(x,\pmb\theta_{k+1},\pmb e_{k+1},\pi,\tau)$ depends on $(\pi,\tau)\sim(\pi^j,\tau^j)_{j=0}^n$ only through their last components $(\pi^j,\tau^j)_{j=k+1}^n$, we have
\begin{eqnarray}
V^k(x,\pmb\theta_k,\pmb e_k)&\geq& I^k(x,\pmb\theta_k,\pmb e_k,\tilde\pi^{\epsilon},\tilde\tau^{\epsilon})\notag\\
&=&\mathbb{E}\bigg[U_{\tau^k}^k(X_{\tau^k}^{k,x,\pi^k},\pmb\theta_k,\pmb e_k)\cdot\alpha_{\tau^k}^k(\pmb\theta_k,\pmb e_k)\notag\\
&+&\int_{\theta_k}^{\tau^k}\int_E I^{k+1}\Big(\Gamma_{\theta_{k+1}}^{k+1}(X_{\theta_{k+1}}^{k,x,\pi^k},\pi_{\theta_{k+1}}^k,e_{k+1}),\pmb\theta_{k+1},\pmb e_{k+1},\tilde\pi^{\epsilon},\tilde\tau^{\epsilon}\Big)\eta(de_{k+1})d\theta_{k+1}\big|\mathcal{F}_{\theta_k}\bigg]\notag\\
&\geq& \mathbb{E}\bigg[U_{\tau^k}^k(X_{\tau^k}^{k,x,\pi^k},\pmb\theta_k,\pmb e_k)\cdot\alpha_{\tau^k}^k(\pmb\theta_k,\pmb e_k)\notag\\
&+&\int_{\theta_k}^{\tau^k}\int_E \bar V^{k+1}\Big(\Gamma_{\theta_{k+1}}^{k+1}(X_{\theta_{k+1}}^{k,x,\pi^k},\pi_{\theta_{k+1}}^k,e_{k+1}),\pmb\theta_{k+1},\pmb e_{k+1}\Big)\eta(de_{k+1})d\theta_{k+1}\big|\mathcal{F}_{\theta_k}\bigg]-\epsilon,\notag
\end{eqnarray}
Therefore, $V^k\geq\bar V^k$, from which the claim of the theorem follows.
\end{proof}

Now let us consider the value function $\mathfrak{V}_0$ in \eqref{4-4}. We have the following result:
\begin{proposition}\label{prop 4-3}
Define value functions $(\bar{\mathfrak{V}}^k)_{k=0}^n$ as
\begin{eqnarray}
\label{4-14} &&\bar{\mathfrak{V}}^n(x,\pmb\theta_n,\pmb e_n)=\esssup_{\pi^n\in\mathcal{A}^n}\essinf_{\tau^n\in\mathcal{T}^n}\mathbb{E}\Big[U_{\tau^n}^n(X_{\tau^n}^{n,x,\pi^n},\pmb\theta_n,\pmb e_n)\cdot \alpha_{\tau^n}^n(\pmb\theta_n,\pmb e_n)\big|\mathcal{F}_{\theta_n}\Big],\ (\pmb\theta_n,\pmb e_n)\in \Delta_n(T)\times E^n,\hspace{-0.2cm}\notag\\
\label{4-15}\ \ \  &&\bar{\mathfrak{V}}^k(x,\pmb\theta_k,\pmb e_k)=\esssup_{\pi^k\in\mathcal{A}^k}\essinf_{\tau^k\in\mathcal{T}^k}\mathbb{E}\Big[U_{\tau^k}^k(X_{\tau^k}^{k,x,\pi^k},\pmb\theta_k,\pmb e_k)\cdot \alpha_{\tau^k}^k(\pmb\theta_k,\pmb e_k)\notag\\
&&\hspace{1cm}+\int_{\theta_k}^{\tau^k}\int_E \bar{\mathfrak{V}}^{k+1}\Big(\Gamma_{\theta_k}^{k+1}(X_{\theta_{k+1}}^{k,x,\pi^k},\pi_{\theta_{k+1}}^k,e_{k+1}),\pmb\theta_k,\theta_{k+1},\pmb e_k,e_{k+1}\Big)\eta(de_{k+1})d\theta_{k+1}\big|\mathcal{F}_{\theta_k}\Big],\notag
\end{eqnarray}
$(\pmb\theta_k,\pmb e_k)\in\Delta_k(T)\times E^k$, for $k=0,\dotso,n-1$. Then $\mathfrak{V}_0(x)=\bar{\mathfrak{V}}^0(x)$.
\end{proposition}
\begin{proof}
Given $\pi\sim(\pi^0,\dotso,\pi^n)$ in $\mathcal{A}_{\mathbb{G}}$, define
\begin{eqnarray}
&& \tilde{\mathfrak{V}}^n(x,\pmb\theta_n,\pmb e_n,\pi)=\essinf_{\tau^n\in\mathcal{T}^n}\mathbb{E}\Big[U_{\tau^n}^n(X_{\tau^n}^{n,x,\pi^n},\pmb\theta_n,\pmb e_n)\cdot \alpha_{\tau^n}^n(\pmb\theta_n,\pmb e_n)\big|\mathcal{F}_{\theta_n}\Big],\ (\pmb\theta_n,\pmb e_n)\in\Delta_n(T)\times E^n,\notag\\
&&\tilde{\mathfrak{V}}^k(x,\pmb\theta_k,\pmb e_k,\pi)=\essinf_{\tau^k\in\mathcal{T}^k}\mathbb{E}\Big[U_{\tau^k}^k(X_{\tau^k}^{k,x,\pi^k},\pmb\theta_k,\pmb e_k)\cdot \alpha_{\tau^k}^k(\pmb\theta_k,\pmb e_k)\notag\\
&&\hspace{1cm}+\int_{\theta_k}^{\tau^k}\int_E \tilde{\mathfrak{V}}^{k+1}\Big(\Gamma_{\theta_k}^{k+1}(X_{\theta_{k+1}}^{k,x,\pi^k},\pi_{\theta_{k+1}}^k,e_{k+1}),\pmb\theta_k,\theta_{k+1},\pmb e_k,e_{k+1},\pi\Big)\eta(de_{k+1})d\theta_{k+1}\big|\mathcal{F}_{\theta_k}\Big],\notag
\end{eqnarray}
$(\pmb\theta_k,\pmb e_k)\in\Delta_k(T)\times E^k$, for $k=0,\dotso,n-1$. From \thref{theorem 4-2}, we have
$$\tilde{\mathfrak{V}}^0(x,\pi)=\inf_{\tau\in\mathcal{T}_{\mathbb{G}}}\mathbb{E}\ U_{\tau}(X_{\tau}^{x,\pi}).$$
Define
 $$\mathfrak{V}^k(x,\pmb\theta_k,\pmb e_k):=\esssup_{\pi\in\mathcal{A}_{\mathbb{G}}}\tilde{\mathfrak{V}}^k(x,\pmb\theta_k,\pmb e_k,\pi),\ (\pmb\theta_k,\pmb e_k)\in\Delta_k(T)\times E^k,\  k=0,\dotso,n.$$
 Then the definition for $\mathfrak{V}^0$ above is consistent with \eqref{4-4}. Following the proof of \thref{theorem 4-2} we can  show $\mathfrak{V}^k=\bar{\mathfrak{V}}^k,\ k=0,\dotso,n$. 
\end{proof}

\section{Application to indifference pricing of American options}\label{sec:indifference-pricing}
In this section, we apply our decomposition method to indifference pricing of American options under multiple default risk. The main results are \thref{theorem 4-4} and \thref{theorem 4-6}, which provide the RBSDE characterization of the indifference prices.

\subsection{Market model}
The model we will use here is similar to that in \cite{pham2}. Let $T\in(0,\infty)$ be the finite time horizon. We assume in the market, there exists at most $n$ default events. Let $\zeta_1,\dotso,\zeta_n$ and $\ell_1,\dotso,\ell_n$ represent the random default times and marks respectively, with $\alpha$ defined in \eqref{3-6} as the probability density.
 For any time $t$, if $\zeta_k\leq t<\zeta_{k+1},\  k=1,\dotso,n-1$ ($t<\zeta_1$ for $k=0$ and $t\geq\zeta_n$ for $k=n$), we say the underlying processes are in the $k$-default scenario. 
 
 We consider a portfolio of $d$-asset with a value process defined by a $d$-dimensional $\mathbb{G}$-optional process $S\sim(S^0,\dotso,S^n)$ from \eqref{3-1}, where $S^k(\pmb\theta_k,\pmb e_k)\in\mathcal{O}_{\mathbb{F}}(\Delta_k,E^k)$ is valued in $\mathbb{R}_+^d$, representing the asset value in the $k$-default scenario, given the past default times $\pmb\zeta_k=\pmb\theta_k$ and the associated marks $\pmb\ell_k=\pmb e_k$, for $k=0,\dotso n$. Suppose the dynamics of the indexed process $S^k$ is given by
\begin{equation}\label{2-c} dS_t^k(\pmb\theta_k,\pmb e_k)=S_t^k(\pmb\theta_k,\pmb e_k)*\left(b_t^k(\pmb\theta_k,\pmb e_k)dt+\sigma_t^k(\pmb\theta_k,\pmb e_k)dW_t\right),\ \ \ t\geq\theta_k,\end{equation}
where $W$ is an $m$-dimensional  $(\mathbb{P},\mathbb{F})$-Brownian motion, $m\geq d$, $b^k$ and $\sigma^k$ are indexed processes in $\mathcal{P}_{\mathbb{F}}(\Delta_k,E^k)$, valued respectively in $\mathbb{R}^d$ and $\mathbb{R}^{d\times m}$. Here, for $x=(x_1,\dotso,x_d)'\in\mathbb{R}^d$, and $y=(y_1,\dotso,y_d)'\in\mathbb{R}^{d\times q}$, the expression $x*y$ denotes the vector $(x_1y_1,\dotso,x_dy_d)'\in\mathbb{R}^{d\times q}$. Equation \eqref{2-c}  can be viewed as an asset model with change of regimes after default events, with coefficient $b^k,\sigma^k$ depending on the past default information.  We make the usual no-arbitrage assumption that there exists an indexed risk premium process $\lambda^k\in\mathcal{P}_{\mathbb{F}}(\Delta_k,E^k)$, such that for all $(\pmb\theta_k,\pmb e_k)\in \Delta_k\times E^k$,
\begin{equation}\label{2-d}\sigma_t^k(\pmb\theta_k,\pmb e_k) \lambda_t^k(\pmb\theta_k,\pmb e_k)=b_t^k(\pmb\theta_k,\pmb e_k),\ \ \ t\geq 0.\end{equation}
Moreover, each default time $\theta_k$ may induce a jump in the asset portfolio, which will be formalized by considering a family of indexed processes $\gamma^k\in \mathcal{P}(\Delta_k,E^k,E)$, valued in $[-1,\infty)$, for $k=0,\dotso,n-1$.  For $(\pmb\theta_k,\pmb e_k)\in\Delta_k\times E^k$ and $e_{k+1}\in E,\ \gamma_{\theta_{k+1}}^k(\pmb\theta_k,\pmb e_k,e_{k+1})$ represents the relative vector jump size on the $d$ assets at time $t=\theta_{k+1}\geq \theta_k$ with a mark $e_{k+1}$, given the past default events $(\pmb\zeta_k,\pmb\ell_k)=(\pmb\theta_k,\pmb e_k)$. In other words, we have:
\begin{equation}\label{2-e} S_{\theta_{k+1}}^{k+1}(\pmb\theta_{k+1},\pmb e_{k+1})=S_{\theta_{k+1}^-}^k(\pmb\theta_k,\pmb e_k)*\left({\bf 1}_d+\gamma_{\theta_{k+1}}^k(\pmb\theta_k,\pmb e_k,e_{k+1})\right),\end{equation}
where ${\bf 1}_d$ is the vector in $\mathbb{R}^d$ with all components equal to $1$.
\begin{remark}
 It is possible that after default times, some assets may not be traded any more. Now suppose that after $k$ defaults, there are $\bar d$ assets still tradable, where $0\leq \bar d\leq d$. Then without loss of generality, we may assume $b^k(\pmb\theta_k,\pmb e_k)=\left(\bar b(\pmb\theta_k,\pmb e_k)\ 0\right),\ \sigma^k(\pmb\theta_k,\pmb e_k)=\left(\bar\sigma^k(\pmb\theta_k,\pmb e_k)\ 0\right),\ \gamma^k(\pmb\theta_k,\pmb e_k,e)=\left(\bar\gamma^k(\pmb\theta_k,\pmb e_k,e)\ 0\right)$, where $\bar b(\pmb\theta_k,\pmb e_k),\ \bar\sigma^k(\pmb\theta_k,\pmb e_k),\ \bar\gamma^k(\pmb\theta_k,\pmb e_k,e)$ are $\mathbb{F}$-predictable processes valued respectively in $\mathbb{R}^{\bar d},\mathbb{R}^{\bar d^k\times m}, \mathbb{R}^{\bar d}$. In this case, we shall also assume that the volatility matrix $\bar\sigma^k(\pmb\theta_k,\pmb e_k)$ is of full rank. we can then define the risk premium
$$\lambda^k(\pmb\theta_k,\pmb e_k)=\bar\sigma^k(\pmb\theta_k,\pmb e_k)'\left(\bar\sigma^k(\pmb\theta_k,\pmb e_k)\bar\sigma^k(\pmb\theta_k,\pmb e_k)'\right)^{-1}\bar b^k(\pmb\theta_k,\pmb e_k),$$
which satisfies \eqref{2-d}.
\end{remark}

An American option of maturity $T$ is modeled by a $\mathbb{G}$-optional process $R\sim(R^0,\dotso,R^n)$ from \eqref{2-2}, where $R_t^k(\pmb\theta_k,\pmb e_k)$ is continuous with respect to $t$, and represents the payoff if the option is exercised at time $t\in[\theta_k,T]$ in the $k$-default scenario, given the past default events $(\pmb\zeta_k,\pmb\ell_k)=(\pmb\theta_k,\pmb e_k)$, for $k=0,\dotso,n$.

A control in the $d$-asset portfolio is a $\mathbb{G}$-predictable process $\pi\sim(\pi^0,\dotso,\pi^n)$, where $\pi^k(\pmb\theta_k,\pmb e_k)\in\mathcal{P}_{\mathbb{F}}(\Delta_k,E^k)$ is valued in a closed set $A^k$ of $\mathbb{R}^d$ containing the zero element, and represents the amount invested continuously in the $d$ assets in the $k$-default scenario, given the past default information $(\pmb\zeta_k,\pmb\ell_k)=(\pmb\theta_k,\pmb e_k)$. An exercise time is a $\mathbb{G}$-stopping time $\tau\sim(\tau^0,\dotso,\tau^n)$ satisfying $\tau\leq T$,  with the decomposition from \prref{prop 2-6}. A trading strategy is a pair of a control and an exercise time. 

For a trading strategy $(\pi,\tau)\sim(\pi^k,\tau^k)_{k=0}^n$, we have the corresponding wealth process $\mathfrak{\mathfrak{X}}\sim(\mathfrak{X}^0,\dotso,\mathfrak{X}^n)$, where $\mathfrak{X}^k(\pmb\theta_k,\pmb e_k)\in\mathcal{O}_{\mathbb{F}}(\Delta_k,E^k)$, representing the wealth controlled by $\pi^k(\pmb\theta_k,\pmb e_k)$ in the price process $S^k(\pmb\theta_k,\pmb e_k)$, given the past default events $(\pmb\zeta_k,\pmb\ell_k)=(\pmb\theta_k,\pmb e_k)$. From \eqref{2-c} we have
\begin{equation}\label{3-e} d\mathfrak{X}_t^k(\pmb\theta_k,\pmb e_k)=\pi_t^k(\pmb\theta_k,\pmb e_k)'\left(b_t^k(\pmb\theta_k,\pmb e_k)dt+\sigma^k(\pmb\theta_k,\pmb e_k)dW_t\right),\ \ t\geq\theta_k.\notag\end{equation}
Moreover, each default time induces a jump in the asset price process, and then also on the wealth process. From \eqref{2-e}, we have
\begin{equation}\label{3-f} \mathfrak{X}_{\theta_{k+1}}^{k+1}(\pmb\theta_{k+1},\pmb e_{k+1})=\mathfrak{X}_{\theta_{k+1}^-}^k(\pmb\theta_k,\pmb e_k)+\pi_{\theta_{k+1}}^k(\pmb\theta_k,\pmb e_k)'\gamma_{\theta_{k+1}}^k(\pmb\theta_k,\pmb e_k,e_{k+1}).\end{equation}

\subsection{Indifference price}
Let $U$ be an exponential utility with risk aversion coefficient $p>0$:
$$U(x)=-\exp(-px),\ \ x\in\mathbb{R},$$
which describes an investor's preference. We will consider two cases. The first case is that the investor can trade the $d$-assets portfolio following control $\pi$, associated to a wealth process $\mathfrak{X}=\mathfrak{X}^{x,\pi}$ with initial capital $\mathfrak{X}_{0^-}=x$. Besides, she holds an American option and can choose to exercise it at any time $\tau$, $\tau\leq T$, to get payoff $R_{\tau}$. So the maximum utility she can get (or as close as she want, if not attainable) is: 
\begin{equation}\label{3-g}V^0(x)=\sup_{\tau}\sup_{\pi}\mathbb{E}\left[U(\mathfrak{X}_{\tau}^{x,\pi}+R_{\tau})\right].\end{equation}
We call $\bar c$ the indifference buying price of the American option, if
$$U(x)=V^0(x-\bar c).$$

The second case is that the investor trades the $d$-asset portfolio following control $\pi$, while shorting an American option. So she has to deliver the payoff $R_{\tau}$ at some exercise time $\tau$, which is chosen by the holder of the option. By considering the worst scenario, the maximum utility she can get (or as close as she want) is:
\begin{equation}\label{3-h}\mathfrak{V}^0(x)=\sup_{\pi}\inf_{\tau}\mathbb{E}\left[U(\mathfrak{X}_{\tau}^{x,\pi}-R_{\tau})\right].\end{equation}
In this case, we call $\underline c$ the indifference selling price of the American option, if
$$U(x)=\mathfrak{V}^0(x+\underline c).$$

\subsection{Indifference buying price}
 In this sub-section, we will focus on the problem \eqref{3-g}. \thref{theorem 4-4} is the main result for this sub-section. 
\begin{definition}\label{def 4-1}
(Admissible trading strategy) A trading strategy $(\pi,\tau)\sim(\pi^k,\tau^k)_{k=0}^n$ is admissible, if for any $(\pmb{\theta}_k,\pmb{e}_k)\in\Delta_k(T)\times E^k$, under the control $\pi^k$,
\begin{itemize}
\item[(a)] $\displaystyle \int_{\theta_k}^{\tau^k}|\pi_t^k(\pmb{\theta}_k,\pmb{e}_k)'b_t^k(\pmb{\theta}_k,\pmb{e}_k)|dt+\int_{\theta_k}^{\tau^k}|\pi_t^k(\pmb{\theta}_k,\pmb{e}_k)'\sigma_t^k(\pmb{\theta}_k,\pmb{e}_k)|^2dt<\infty,  \ \ \text{a.s.,}\ \ k=0,\dotso,n,$
\item[(b)] the family $\displaystyle \Big\{ U(\mathfrak{X}_{\tau\wedge\tau^k}^k(\pmb{\theta}_k,\pmb{e}_k)):\ \ \tau \text{ is any }\mathbb{F}-\text{stopping time valued in }[\theta_k,T]\Big\}$ is uniformly integrable, i.e., $U(\mathfrak{X}^k_{\cdot\wedge\tau^k}(\pmb{\theta}_k,\pmb{e}_k))$ is of class (D), for $k=0,\dotso,n$,
\item[(c)] $\displaystyle \mathbb{E}\left[\int_{\theta_k}^{\tau^k}\int_E(-U)\left(\mathfrak{X}_s^k(\pmb{\theta}_k,\pmb{e}_k)+\pi_s^k(\pmb{\theta}_k,\pmb{e}_k)'\gamma_s^k(\pmb{\theta}_k,\pmb{e}_k,e)\right)\eta(de)ds\right]<\infty$, \ for $k=0,\dotso,n-1$.
\end{itemize}
\end{definition}
The notation $\mathcal{A}_\mathbb{G},\ \mathcal{T}_\mathbb{G},\ \mathcal{A}^k$ and $\mathcal{T}^k$ from Section 5 are now specified by the above definition. From \thref{theorem 4-2}, $V^0$ in \eqref{3-g} can be calculated by the following backward induction: 
\begin{eqnarray}
\label{4-a}&&V^n(x,\pmb{\theta}_n,\pmb{e}_n)= \esssup_{\tau^n\in\mathcal{T}^n} \esssup_{\pi^n\in\mathcal{A}^n}\mathbb{E}\Big[U(\mathfrak{X}^{n,x}_{\tau^n}+H_{\tau^n}^n)|\mathcal{F}_{\theta_n}\Big],\ \ (\pmb\theta_n,\pmb e_n)\in\Delta_n(T)\times E^n,\\
\label{4-b}&&V^k(x,\pmb{\theta}_k,\pmb{e}_k)= \esssup_{\tau^k\in\mathcal{T}^k} \esssup_{\pi^k\in\mathcal{A}^k}\mathbb{E}\Big[U(\mathfrak{X}^{k,x}_{\tau^k}+H_{\tau^k}^k)\\
&&\hspace{1cm}+\int_{\theta_k}^{\tau^k}\int_EV^{k+1}(\mathfrak{X}_{\theta_{k+1}}^{k,x}+\pi_{\theta_{k+1}}^k\cdot\gamma_{\theta_{k+1}}^k(e_{k+1}),\pmb{\theta}_{k+1},\pmb{e}_{k+1})\eta(de_{k+1})d\theta_{k+1}|\mathcal{F}_{\theta_k}\Big],\notag
\end{eqnarray}
$(\pmb\theta_k,\pmb e_k)\in\Delta_k(T)\times E^k$, for $k=0,\dotso,n-1$, where
$$H^k:=R^k-\frac{1}{p}\ln\alpha^k,$$ in which $\alpha^k$ is given by \eqref{eq:alphak}.

\subsubsection{ Backward recursive system of RBSDEs}
Following \cite{Ying1}, we expect the value function to be of the following form:
\begin{equation}\label{4-aa}V^k(x,\pmb\theta_k,\pmb e_k)=U\left(x+Y_{\theta_k}^k(\pmb\theta_k,\pmb e_k)\right),\end{equation}
where $Y^k(\pmb\theta_k,\pmb e_k)$ is an $\mathbb{F}$-adapted process, satisfying the RBSDE \textbf{eq}$(H^k(\pmb\theta_k,\pmb e_k),f^k)_{\theta_k\leq t\leq T}$, with $f^k$ defined as
\begin{eqnarray}\label{4-i} f^k(t,y,z,\pmb{\theta}_k,\pmb{e}_k)=\inf_{\pi\in A^k}g^k(\pi,t,y,z,\pmb{\theta}_k,\pmb{e}_k),\end{eqnarray}
where
\begin{eqnarray}
\label{2vv} &&g^k(\pi,t,y,z,\pmb{\theta}_k,\pmb{e}_k)=\frac{p}{2}\left| z-\sigma_t^k(\pmb{\theta}_k,\pmb{e}_k)'\pi\right|^2-b_t^k(\pmb{\theta}_k,\pmb{e}_k)'\pi\notag\\
&&\hspace{2.5cm}+\frac{1}{p}U(-y)\displaystyle\int_EU\left(\pi\cdot\gamma_t^k(\pmb{\theta}_k,\pmb{e}_k,e)+Y_t^{k+1}(\pmb{\theta}_k,t,\pmb{e}_k,e)\right)\eta(de)\notag\\
&&\hspace{2.5cm}=-\lambda_t^k(\pmb{\theta}_k,\pmb{e}_k)\cdot z-\frac{1}{2p}|\lambda_t^k(\pmb{\theta}_k,\pmb{e}_k)|^2+\frac{p}{2}\left|z+\frac{1}{p}\lambda_t^k(\pmb{\theta}_k,\pmb{e}_k)-\sigma_t^k(\pmb{\theta}_k,\pmb{e}_k)'\pi\right|^2\notag\\
&&\hspace{2.5cm}+\frac{1}{p}U(-y)\int_EU\big(\pi\cdot\gamma_t^k(\pmb{\theta}_k,\pmb{e}_k,e)+Y_t^{k+1}(\pmb{\theta}_k,t,\pmb{e}_k,e)\big)\eta(de),\notag
\end{eqnarray}
for $k=0,\dotso,n-1$, and 
\begin{eqnarray}
\label{2ww}&&g^n(\pi,t,y,z,\pmb{\theta}_n,\pmb{e}_n)=\frac{p}{2}\left| z-\sigma_t^n(\pmb{\theta}_n,\pmb{e}_n)'\pi\right|^2-b_t^n(\pmb{\theta}_n,\pmb{e}_n)'\pi\notag\\
&&\hspace{2cm}=-\lambda_t^n(\pmb{\theta}_n,\pmb{e}_n)\cdot z-\frac{1}{2p}|\lambda_t^n(\pmb{\theta}_n,\pmb{e}_n)|^2+\frac{p}{2}\left|z+\frac{1}{p}\lambda_t^n(\pmb\theta_n,\pmb{e}_n)-\sigma_t^n(\pmb{\theta}_n,\pmb{e}_n)'\pi\right|^2.\notag
\end{eqnarray}
In the next two subsections, we will show that: (a) The backward recursive system of RBSDEs admits a solution; (b) The solution characterizes the values of $(V^k)$, i.e., \eqref{4-aa} holds.
\subsubsection{Existence to the recursive system of RBSDEs}\label{sec:exsrcsys}
We make the following boundedness assumptions \textbf{(HB)}:
\begin{itemize}
\item[(i)] The risk premium is bounded uniformly with respect to its indices: there exists a constant $C>0$, such that for any $k=0,\dotso,n$, $(\pmb{\theta}_k,\pmb{e}_k)\in\Delta_k(T)\times E^k$, $t\in[\theta_k,T]$,
$$|\lambda_t^k(\pmb{\theta}_k,\pmb{e}_k)|\leq C, \ \ \ \text{a.s.}$$
\item[(ii)] The indexed random variables $(H_t^k)_k$ are bounded uniformly in time and their indices: there exists a constant $C>0$ such that for any $k=0,\dotso,n$, $(\pmb{\theta}_k,\pmb{e}_k)\in\Delta_k(T)\times E^k$, $t\in[\theta_k,T]$,
$$|H_t^k(\pmb{\theta}_k,\pmb{e}_k)|\leq C,\ \ \ \text{a.s.}$$
\end{itemize}

\begin{theorem}\label{theorem 6-3}
Under \textbf{(HB)}, there exists a solution $(Y^k,Z^k,K^k)_{k=0}^n\in\prod_{k=0}^n\mathcal{S}_c^\infty(\Delta_k(T),E^k)\times\mathbf{\bf L}_W^2$ $(\Delta_k(T),E^k)\times\mathbf{A}(\Delta_k(T),E^k) $ to the recursive system of indexed RBSDEs \textbf{eq}$(H^k(\pmb\theta_k,\pmb e_k),f^k)_{\theta_k\leq t\leq T}$, $k=0,\dotso,n$.
\end{theorem}
\begin{proof} We prove the result by a backward induction on $k=0,\dotso,n$. The positive constant $C$ may vary from line to line, but is always independent of $(t,\omega,\pmb\theta_k,\pmb e_k)$. We will often omit the dependence of $(t,\omega,y,z,\pmb\theta_k,\pmb e_k)$ in related functions.
\vspace{0.2cm}

\noindent (a) For $k=n$. Under \textbf{(HB)}, $|f^n|\leq C(|z|^2+1)$. By Theorem 1 in \cite{Kobylanski1}, there exists a solution $\big(Y^n(\pmb{\theta}_n,\pmb{e}_n),Z^n(\pmb{\theta}_n,\pmb{e}_n),K^n$ $(\pmb{\theta}_n,\pmb{e}_n)\big)\in\mathcal{S}_c^\infty[\theta_n,T]\times\mathbf{\bf L}_W^2[\theta_n,T]\times\mathbf{A}[\theta_n,T]$ for \textbf{eq}$\displaystyle(H^n,f^n)_{\theta_n\leq t\leq T}$, satisfying $|Y^n|\leq C$. Moreover, the measurability of $(Y^n,Z^n)$ with respect to $(\pmb\theta_n,\pmb e_n)$ follows from the measurability of $H^n$ and $f^n$ (see  Appendix C in \cite{Lim1} and use the fact that the solution to the RBSDE can be eventually approximated by the solutions to BSDEs). Therefore, $(Y^n,Z^n,K^n)\in\mathcal{S}_c^\infty(\Delta_n(T),E^n)\times\mathbf{\bf L}_W^2(\Delta_n(T),E^n)\times\mathbf{A}(\Delta_n(T),E^n)$.
\vspace{0.2cm}
  
\noindent (b) For $k\in\{0,1,\dotso,n-1\}$. Assume there exists $(Y^{k+1},Z^{k+1},K^{k+1})\in\mathcal{S}_c^\infty(\Delta_{k+1}(T),E^{k+1})\times\mathbf{L}_W^2(\Delta_{k+1}(T),E^{k+1})\times\mathbf{A}(\Delta_{k+1}(T),E^{k+1})$ satisfying \textbf{eq}$(H^{k+1},f^{k+1})$. Since $Y^{k+1}\in\mathcal{P}_\mathbb{F}(\Delta_{k+1},$ $E^{k+1})$, the generator in \eqref{4-i} is well defined. In order to overcome the technical difficulties coming from the exponencial term in $U(-y)$, we first consider the truncated generator
\begin{equation}\label{2j}f^{k,N}(t,y,z,\pmb{\theta}_k,\pmb{e}_k)=\inf_{\pi\in A^k}g^k(\pi,t,N\wedge y,z,\pmb{\theta}_k,\pmb{e}_k).\notag\end{equation}
Then there exists a positive constant $C_N$ independent of $(\pmb\theta_k,\pmb e_k)$, such that $|f^{k,N}|\leq C_N(1+z^2)$. Applying Theorem 1 in \cite{Kobylanski1}, there exists a solution $(Y^{k,N}, Z^{k,N},K^{k,N})\in\mathcal{S}_c^\infty(\Delta_k(T),E^k)\times\mathbf{\bf L}_W^2(\Delta_k(T),E^k)\times\mathbf{A}(\Delta_k(T),E^k)$ to \textbf{eq}$(H^k,f^{k,N})$. 

Now we will show that $Y^{k,N}$ has a uniform upper bound. Consider the generator
$$\bar{f}^k(t,y,z,\pmb{\theta}_k,\pmb{e}_k):=-\lambda_t^k(\pmb{\theta}_k,\pmb{e}_k)\cdot z-\frac{1}{2p}|\lambda_t^k(\pmb{\theta}_k,\pmb{e}_k)|^2,$$
which satisfies the Lipschitz condition in $(y,z)$, uniformly in $(t,\omega)$. Then by Theorem 5.2 in \cite{El1}, there exists a unique solution  $\big(\bar Y^k(\pmb{\theta}_k,\pmb{e}_k),\bar Z^k(\pmb{\theta}_k,\pmb{e}_k),\bar K^k(\pmb{\theta}_k,\pmb{e}_k)\big)\in\mathcal{S}_c^\infty[\theta_k,T]\times\mathbf{\bf L}_W^2[\theta_k,T]\times\mathbf{A}[\theta_k,T]$ satisfying $|\bar Y^k|\leq C$ (see Theorem 1 in \cite{Kobylanski1} for the boundedness). Applying Lemma 2.1(comparison) in \cite{Kobylanski1}, we get $Y^{k,N}\leq\bar Y^k$. Hence, $Y^{k,N}$ has a uniform upper bound independent of $N$ and $(\pmb\theta_k,\pmb e_k)$. Therefore, for $N$ large enough, we can remove \lq\lq$N$\rq\rq\ in the truncated generator $f^{k,N}$, i.e., $(Y^{k,N},Z^{k,N},K^{k,N})$ solves \textbf{eq}$(H^k,f^k)$ for large enough $N$.
\end{proof}

\subsubsection{RBSDE characterization by verification theorem}
\begin{theorem}\label{theorem 4-4}
The value functions $(V^k)_{k=0}^n$, defined in \eqref{4-a} and \eqref{4-b}, are given by
\begin{equation}\label{2rr} V^k(x,\pmb{\theta}_k,\pmb{e}_k)=U\big(x+Y_{\theta_k}^k(\pmb{\theta}_k,\pmb{e}_k)\big),\end{equation}
for $\forall x\in\mathbb{R},\ (\pmb{\theta}_k,\pmb{e}_k)\in\Delta_k(T)\times E^k$, where $(Y^{k},Z^{k},K^{k})_{k=0}^n\in\mathop{\prod}\limits_{k=0}^{n}\mathcal{S}_c^\infty$ $(\Delta_{k}(T),E^{k})\times\mathbf{\bf L}_W^2(\Delta_{k}(T),E^{k})\times\mathbf{A}(\Delta_{k}(T),E^{k})$ is a solution of the RBSDE system \textbf{eq}$(H^k,f^k),\ k=0,\dotso,n$. Moreover, there exists an optimal trading strategy $(\pi,\tau)\sim(\hat\pi^k,\hat\tau^k)_{k=0}^n$ described by:
\begin{equation}\label{2ss}\hat\pi_t^k(\pmb{\theta}_k,\pmb{e}_k)\in\argmin_{\pi\in A^k}g^k\big(\pi,t,Y_t^k(\pmb\theta_k,\pmb e_k),Z_t^k(\pmb\theta_k,\pmb e_k),\pmb{\theta}_k,\pmb{e}_k\big),\notag\end{equation}
for $t\in[\theta_k,T]$, and
\begin{equation}\label{2uu}\hat\tau^k(\pmb{\theta}_k,\pmb{e}_k):=\inf\left\{t\geq \theta_k: \ Y_t^k(\pmb{\theta}_k,\pmb{e}_k)=H_t^k(\pmb{\theta}_k,\pmb{e}_k)\right\},\end{equation}
for $(\pmb{\theta}_k,\pmb{e}_k)\in\Delta_k(T)\times E^k$, a.s., $k=0,\dotso,n$.
\end{theorem}
\begin{proof}
\textbf{Step 1:} \textbf{We will show} 
\begin{equation}\label{2u}U(x+Y^k_{\theta_k}(\pmb{\theta}_k,\pmb{e}_k))\geq V^k(x,\pmb{\theta}_k,\pmb{e}_k),\text{{ }{ }}k=0,\dotso,n.\end{equation}
Let  $(Y^{k},Z^{k},K^{k})\in\mathcal{S}_c^\infty(\Delta_{k}(T),E^{k})\times\mathbf{\bf L}_W^2(\Delta_{k}(T),E^{k})\times\mathbf{A}(\Delta_{k}(T),E^{k})$ be a solution of the RBSDE system. For $(\nu^k,\tau^k)\in\mathcal{A}^k\times\mathcal{T}^k$, $x\in\mathbb{R}, \ (\pmb{\theta}_k,\pmb{e}_k)\in\Delta_k(T)\times E^k$ and $t\geq \theta_k$, and define
\begin{eqnarray}
\label{2v}\xi_t^k(x,\pmb{\theta}_k,\pmb{e}_k,\nu^k)&:=&U(\mathfrak{X}_t^{k,x}+Y_t^k(\pmb{\theta}_k,\pmb{e}_k))+\int_{\theta_k}^t\int_EU\Big(\mathfrak{X}_r^{k,x}+\nu_r^k\cdot\gamma_r^k(\pmb{\theta}_{k},e_{k},e)\notag\\
&&+Y_r^{k+1}(\pmb\theta_k,r,\pmb{e}_k,e)\Big)\eta(de)dr, \ \ \ k=0,\dotso,n-1,\notag\\
\label{2w}\xi_t^n(x,\pmb{\theta}_n,\pmb{e}_n,\nu^n)&:=&U(\mathfrak{X}_t^{n,x}+Y_t^n(\pmb{\theta}_n,\pmb{e}_n)).\notag\end{eqnarray}
Applying It$\hat{\text{o}}$'s formula, we get for $k=0,\dotso,n$, 
\begin{equation}\label{2x}\hspace{-2.1cm}\xi_t^k(x,\pmb{\theta}_k,\pmb{e}_k,\nu^k)=pU\left(\mathfrak{X}_t^{k,x}+Y_t^k(\pmb{\theta}_k,\pmb{e}_k)\right) \Big[\big(-f^k(t,Y_t^k,Z_t^k,\pmb{\theta}_k,\pmb{e}_k)\notag\end{equation}
$$\hspace{4cm}+g^k(\nu_t^k,t,Y_t^k,Z_t^k,\pmb{\theta}_k,\pmb{e}_k)\big)dt+dK_t^k(\pmb{\theta}_k,\pmb{e}_k)+(Z_t^k-\sigma_t^k(\pmb{\theta}_k,\pmb{e}_k)'\nu_t^k)\cdot dW_t\Big],$$
$\displaystyle f^k(\cdot)=\inf_{\pi\in A^k}g^k(\pi,\cdot)$ implies $\displaystyle\left\{\xi_s^k(x,\pmb{\theta}_k,\pmb{e}_k,\nu^k)\right\}_{\theta_k\leq t\leq T}$ is a local super-martingale, for $k=0,\dotso,n$. Since $Y^k$ and $Y^{k+1}$ are essentially bounded, and $\xi_{t\wedge\tau^k\wedge\rho_m}^k(x,\pmb{\theta}_k,\pmb{e}_k,\nu^k)$ is uniformly integrable, by considering a localizing sequence of stopping times, we can show $\left\{\xi_{t\wedge\tau^k}^k(x,\pmb{\theta}_k,\pmb{e}_k,\nu^k)\right\}_{\theta_k\leq t\leq T}$ is a super-martingale. Consider when $k=n$. Since $Y^n\geq H^n$, we have
\begin{equation}\label{2cc}U\left(x+Y_{\theta_n}^n(\pmb{\theta}_n,\pmb{e}_n)\right)\geq\mathbb{E}\left[U\left(\mathfrak{X}_{\tau^n}^{n,x}+H_{\tau^n}^n(\pmb{\theta}_n,\pmb{e}_n)\right)|\mathcal{F}_{\theta_n}\right].\end{equation}
Therefore, \eqref{2u} holds for $k=n$. Similarly, it holds for $k=0,\dotso,n-1$.\\
\noindent\textbf{Step 2:} \textbf{$\int_{\theta_k}^\cdot Z_s^k(\pmb{\theta}_k,\pmb{e}_k)\cdot dW_s$ is a BMO-martingale.} Apply It$\hat{\text{o}}$'s formula to $\exp(-qY_t^k(\pmb{\theta}_k,\pmb{e}_k))$ with $q>p$ and  any $\mathbb{F}$-stopping time $\tau$ valued in $[\theta_k,T]$,
\begin{eqnarray}
&&\hspace{-1.5cm}\frac{1}{2}q(q-p)\mathbb{E}\left[\displaystyle\int_\tau^T \exp\left(-qY_t^k(\pmb{\theta}_k,\pmb{e}_k)\right)|Z_t^k(\pmb{\theta}_k,\pmb{e}_k)|^2dt\Big|\mathcal{F}_\tau\right]\notag\\
&&=q\mathbb{E}\left[\int_\tau^T \exp\left(-qY_t^k(\pmb{\theta}_k,\pmb{e}_k)\right)\left(f^k(t,Y_t^k,Z_t^k,\pmb{\theta}_k,\pmb{e}_k)-\displaystyle\frac{p}{2}|Z_t^k|^2\right)dt\Big|\mathcal{F}_\tau\right]\notag\\
&&+\mathbb{E}\left[ \exp\left(-qY_T^k(\pmb{\theta}_k,\pmb{e}_k)\right)-\exp\left(-qY_{\tau}^k(\pmb{\theta}_k,\pmb{e}_k)\right)\Big|\mathcal{F}_{\tau}\right]\notag\\
&&-q\mathbb{E}\left[\int_\tau^T\exp\left(-qY_t^k(\pmb{\theta}_k,\pmb{e}_k)\right)dK_t^k(\pmb{\theta}_k,\pmb{e}_k)\Big|\mathcal{F}_\tau\right].\notag
\end{eqnarray}
Since $|f^k(t,y,z,\pmb{\theta}_k,\pmb{e}_k)|\leq\displaystyle\frac{p}{2}|z|^2-CU(-y)$, $dK^k\geq 0$ and $Y^k$ is bounded, we have
\begin{equation}\hspace{-4.5cm}\frac{1}{2}q(q-p)\mathbb{E}\left[\int_\tau^T \exp\left(-qY_t^k(\pmb{\theta}_k,\pmb{e}_k)\right)|Z_t^k(\pmb{\theta}_k,\pmb{e}_k)|^2dt\Big|\mathcal{F}_\tau\right]\notag\end{equation}
$$\leq qC \mathbb{E}\left[\int_\tau^T\exp\left(-qY_t^k(\pmb{\theta}_k,\pmb{e}_k)\right)dt\Big|\mathcal{F}_\tau\right]+C.$$
By choosing $q$ large enough, we have
$$\mathbb{E}\left[\int_\tau^T \left| Z_s^k(\pmb{\theta}_k,\pmb{e}_k)\right|^2 ds\Big|\mathcal{F}_\tau\right]\leq C,$$
which implies  $\int_{\theta_k}^\cdot Z_s^k(\pmb{\theta}_k,\pmb{e}_k)\cdot dW_s$ is a BMO-martingale.\\
\noindent \textbf{Step 3:} \textbf{Adimissibility of $(\hat\pi^k,\hat\tau^k)$.} For $k=0,\dotso,n$, define function $\hat g^k$ by
\begin{equation}\hat g^k(\pi,t,\omega,\pmb{\theta}_k,\pmb{e}_k)=g^k\big(\pi,t,Y_t^k(\pmb\theta_k,\pmb e_k),Z_t^k(\pmb\theta_k,\pmb e_k),\pmb{\theta}_k,\pmb{e}_k\big).\notag\end{equation}
We can show that the map $(\pi,t,\omega,\pmb{\theta}_k,\pmb{e}_k)\rightarrow \hat g^k(\pi,t,\omega,\pmb{\theta}_k,\pmb{e}_k)$ is $\mathcal{B}(\mathbb{R}^d)\otimes\mathcal{P}_\mathbb{F}\otimes\mathcal{B}(\Delta_k)\otimes\mathcal{B}(E_k)$-measurable. Now for $k=0,\dotso,n$, $(\pmb{\theta}_k,\pmb{e}_k)\in\Delta_k(T)\times E^k$, if either $\sigma^k(\pmb{\theta}_k,\pmb{e}_k)=0$ or $\gamma^k(\pmb{\theta}_k,\pmb{e}_k,e)=0$, then the continuous function $\pi\rightarrow\hat g^k(\pi,t,\omega,\pmb{\theta}_k,\pmb{e}_k)$ attains trivially its infimum of $\hat g^k$ when $\pi=0$.  Otherwise, $\sigma^k(\pmb{\theta}_k,\pmb{e}_k)$ and $\gamma^k(\pmb{\theta}_k,\pmb{e}_k,e)$ are in the form $\sigma^k(\pmb{\theta}_k,\pmb{e}_k)=(\bar \sigma^k(\pmb{\theta}_k,\pmb{e}_k),0),$ $ \ \gamma^k(\pmb{\theta}_k,\pmb{e}_k)=(\bar \gamma(\pmb{\theta}_k,\pmb{e}_k), 0)$ for some full rank matrix $\bar\sigma^k(\pmb{\theta}_k,\pmb{e}_k)$. In this case, we let $(\bar \pi,0)=(\sigma^k)'\cdot\pi$, then we get
\begin{eqnarray}
\label{2mm}&&\bar g^k(\bar\pi,t,\omega,\pmb{\theta}_k,\pmb{e}_k):=\hat g^k(\pi,t,\omega,\pmb{\theta}_k,\pmb{e}_k)=\frac{p}{2}\left|Z_t^k(\pmb{\theta}_k,\pmb{e}_k)+\frac{1}{p}\lambda_t^k(\pmb{\theta}_k,\pmb{e}_k)-\bar \pi\right|^2\notag\\
&&\hspace{2.5cm}+\frac{1}{p}U(-Y_t^k)\int_EU\left( ((\bar \sigma^k )')^{-1}\cdot\bar \pi\cdot \bar \gamma_t^k(e)+Y_t^{k+1}(\pmb{\theta}_k,t,\pmb{e}_k,e)\right)\eta(de),\notag
\end{eqnarray}
for $k=0,\dotso,n-1$, and
\begin{equation}\label{2nn}\bar g^n(\bar\pi,t,\omega,\pmb{\theta}_n,\pmb{e}_n):=\hat g^n(\pi,t,\omega,\pmb{\theta}_n,\pmb{e}_n)=\frac{p}{2}\left|Z_t^n(\pmb{\theta}_n,\pmb{e}_n)+\frac{1}{p}\lambda_t^n(\pmb{\theta}_n,\pmb{e}_n)-\bar \pi\right|^2.\notag\end{equation}
Since
\begin{equation} \bar g^k(0,t,\omega,\pmb{\theta}_k,\pmb{e}_k)<\liminf_{|\bar \pi|\rightarrow\infty}\bar g^k(\bar \pi,t,\omega,\pmb{\theta}_k,\pmb{e}_k), \notag\end{equation}
the continuous function $\bar \pi\rightarrow\bar g^k(\bar\pi,\,t,\omega,\pmb{\theta}_k,\pmb{e}_k)$ attains its infimum over the closed set $(\sigma_t^k)'A^k$, and thus the function $\pi\rightarrow \hat g^k(\pi,t,\omega,\pmb{\theta}_k,\pmb{e}_k)$ attains its infimum over $A^k(\pmb{\theta}_k,\pmb{e}_k)$. For $k=0,\dotso,n$, using a measurable selection argument (see \cite{Wagner1}), one can show that there exists $\hat\pi^k\in\mathcal{P}_{\mathbb{F}}(\Delta_k,E^k)$, such that
\begin{equation}\hat\pi_t^k(\pmb{\theta}_k,\pmb{e}_k)\ \in\ \argmin_{\pi\in A^k(\pmb{\theta}_k,\pmb{e}_k)}\hat g^k(\pi,t,\pmb{\theta}_k,\pmb{e}_k),\ \ \ \theta_k\leq t\leq T,\ \ \text{a.s.} \notag\end{equation}
Consider $\hat\tau^k$ defined in \eqref{2uu}. For $k=0,\dotso,n$, define $\tilde\tau^k(\pmb\zeta_k,\pmb\ell_k)$ as
$$\tilde\tau^k:=\Big(\inf\{t\geq\zeta_k:\ Y^k(\pmb\zeta_k,\pmb\ell_k)=H^k(\pmb\zeta_k,\pmb\ell_k)\}\wedge T\Big)\cdot 1_{\{\zeta_k\leq T\}}+\zeta_k\cdot 1_{\{\zeta_k>T\}}.$$
We can show that $\tilde\tau^k(\pmb\zeta_k,\pmb\ell_k)$ is a $\mathbb{G}^k$ stopping time satisfying $\tilde\tau^k(\pmb\zeta_k,\pmb\ell_k)\geq\zeta_k$ and $\{\tilde\tau^k(\pmb\zeta_k,\pmb\ell_k)\leq T\}=\{\zeta_k\leq T\}$. And given $(\pmb\zeta_k,\pmb\ell_k)=(\pmb\theta_k,\pmb e_k)\in\Delta_k(T)\times E^k$, $\tilde\tau^k(\pmb\theta_k,\pmb e_k)=\hat\tau^k(\pmb\theta_k,\pmb e_k)$.
Now we will show that $(\hat\pi^k,\hat\tau^k)_{k=0}^n$ is admissible in the sense of \deref{def 4-1}.
\medskip

\noindent (a) Since $\hat g^k(\hat\pi_t^k,t,\pmb{\theta}_k,\pmb{e}_k)\leq\hat g^k(0,t,\pmb{\theta}_k,\pmb{e}_k)$, there exists a constant $C>0$, such that
$$|\sigma_t^k(\pmb{\theta}_k,\pmb{e}_k)'\hat\pi_t^k(\pmb{\theta}_k,\pmb{e}_k)|\leq C(1+|Z_t^k(\pmb{\theta}_k,\pmb{e}_k)|),\ \ \theta_k\leq t\leq T,\ \ \text{a.s.},$$
for all $(\pmb{\theta}_k,\pmb{e}_k)\in\Delta_k(T)\times E^k,\ k=0,\dotso,n$. Since $Z^k\in{\bf L}_W^2(\Delta_k,E^k)$ and because of  \textbf{(HB)}(i), $(\hat\pi^k,\hat\tau^k)_{k=0}^n$ satisfies condition (a) in \deref{def 4-1}.
\medskip

\noindent (b) Denote by $\hat{\mathfrak{X}}^{k.x}$ the wealth process controlled by $\hat\pi^k$, starting from $x$ at time $\theta_k$. We have
\begin{equation}\label{a4}f^k(t,Y_t^k,Z_t^k,\pmb{\theta}_k,\pmb{e}_k)=g^k(\hat\pi_t^k,t,Y_t^k,Z_t^k,\pmb\theta_k,\pmb e_k),\notag\end{equation}
for $k=0,\dotso,n$. Then for $\theta_k\leq t\leq T$,
$$U(\hat{\mathfrak{X}}_t^{k,x}+Y_t^k)=U(x+Y_{\theta_k}^k)\ \mathcal{E}_t^k\left(p(Z^k-(\sigma^k)'\hat\pi^k)\right)R_t^k,$$
where
$$\mathcal{E}_t^k\left(p(Z^k-(\sigma^k)'\hat\pi^k)\right)=\exp\left(p\int_{\theta_k}^t(Z_s^k-(\sigma_s^k)'\hat\pi_s^k)\cdot dW_s-\frac{p^2}{2}\int_{\theta_k}^t|Z_s^k-(\sigma_s^k)'\hat\pi_s^k|^2ds\right),$$
for $k=0,\dotso ,n$, and
$$R_t^k=\exp\left(pK_t^k-\int_{\theta_k}^tU(-Y_s^k)\int_EU\left(\hat\pi_t^k\cdot\gamma_t^k(\pmb{\theta}_k,\pmb{e}_k)+Y_t^{k+1}(\pmb{\theta}_k,t,\pmb{e}_k,e)\right)\eta(de)ds\right),$$
for $k=0,\dotso,n-1$ and $R_t^n=\exp (pK_t^n)$. From Step 2, $\int_{\theta_k}^\cdot p(Z^k-(\sigma^k)\hat\pi^k)\cdot dW$ is a BMO-martingale and hence $\mathcal{E}_{\cdot\wedge\hat\tau^k}^k\left(p(Z^k-(\sigma^k)'\hat\pi^k)\right)$ is of class (D). Moreover, since $U$ is nonpositive and $K_t^k=0$ when $t\leq \hat\tau^k$, we have $|R_{\cdot\wedge\hat\tau^k}|\leq 1$, and thus $U(\hat{\mathfrak{X}}_{t\wedge\hat\tau^k}^{k,x}+Y_{t\wedge\hat\tau^k}^k)$ is of class (D). So is $U(\hat{\mathfrak{X}}^{k,x}_{\cdot\wedge\hat\tau^k})$ since $Y^k$ is essentially bounded.
\medskip

\noindent (c) Because $dK_t^k=0$ when $t\leq \hat\tau^k$, the process $\xi_{\cdot\wedge\hat\tau^k}^k(x,\pmb{\theta}_k,\pmb{e}_k,e)$ defined in Step 1 under control $\hat\pi^k$ is a local martingale. By considering a localizing $\mathbb{F}$-stopping time sequence $(\rho_m)_m$ valued in $[\theta_k,T]$, we obtain:
\begin{eqnarray}
&&\mathbb{E}\left[\int_{\theta_k}^{\hat\tau^k\wedge\rho_m}\int_E(-U)\left(\hat{\mathfrak{X}}_t^{k,x}+\hat\pi_t^k\cdot\gamma_t^k(\pmb{\theta}_k,\pmb{e}_k,e)+Y_t^{k+1}(\pmb{\theta}_k,t,\pmb{e}_k,e)\right)\eta(de)dt\right]\notag\\
&&=\mathbb{E}\left[U(\hat{\mathfrak{X}}_{\hat\tau^k\wedge\rho_m}^{k,x}+Y_{\hat\tau^k\wedge\rho_m}^k)-U(x+Y_{\theta_k}^k)\right]\leq\mathbb{E}\left[-U(x+Y_{\theta_k}^k)\right],\notag
\end{eqnarray}
By Fatou's lemma, we get Condition (c) in \deref{def 4-1} holds.\\
\noindent\textbf{Step 4:} \textbf{We will show \eqref{2rr} holds and $(\hat\pi^k,\hat\tau^k)_{k=0}^n$ is an optimal trading strategy.} Consider when $k=n$. By the admissibility of $(\hat\pi^n,\hat\tau^n)$, the local martingale $\xi_{t\wedge \hat\tau^n}$ under the control $\hat\pi^n$ is a martingale. Thus, 
\begin{equation}\label{2pp}U(x+Y_{\theta_n}^n)=\mathbb{E}\left[U(\hat{\mathfrak{X}}_{\hat \tau^n}^{n,x}+H_{\hat\tau^n}^n)\Big| \mathcal{F}_{\theta_n}\right].\notag\end{equation}
Along with \eqref{2cc} this results in
\begin{eqnarray}
V^n(x,\pmb{\theta}_n,\pmb{e}_n)&=&\esssup_{\tau^n\in\mathcal{T}^n}\esssup_{\pi^n\in\mathcal{A}^n}\mathbb{E}\left[U(\mathfrak{X}_{\tau^n}^{n,x}+H_{\tau^n}^n(\pmb{\theta}_n,\pmb{e}_n)\big|\mathcal{F}_{\theta_n}\right]\leq U(x+Y_{\theta_n}^n(\pmb{\theta}_n,\pmb{e}_n))\notag\\
&=&\mathbb{E}\left[U(\hat{\mathfrak{X}}_{\hat\tau^n}^{n,x}+H_{\hat\tau^n}^n(\pmb{\theta}_n,\pmb{e}_n))\big|\mathcal{F}_{\theta_n}\right]\leq V^n(x,\pmb{\theta}_n,\pmb{e}_n),\notag
\end{eqnarray}
which implies \eqref{2rr} for $k=n$ and the optimality of $(\hat\pi^n,\hat\tau^n)$. We can show \eqref{2rr} and the optimality of $(\hat\pi^k,\hat\tau^k)$ for $k=0,\dotso,n-1$, similarly using \eqref{4-b}.
\end{proof}

\subsection{Indifference selling price}

 In this sub-section, we consider the problem \eqref{3-h}, and \thref{theorem 4-6} is the main result.
\begin{definition}
(Admissible trading strategy) A trading strategy $(\pi,\tau)\sim(\pi^k,\tau^k)_{k=0}^n$ is admissible, if for any $(\pmb{\theta}_k,\pmb{e}_k)\in\Delta_k(T)\times E^k$, under the control $\pi^k$,
\begin{itemize}
\item[(a)] $\displaystyle \int_{\theta_k}^T|\pi_t^k(\pmb{\theta}_k,\pmb{e}_k)'b_t^k(\pmb{\theta}_k,\pmb{e}_k)|dt+\int_{\theta_k}^T|\pi_t^k(\pmb{\theta}_k,\pmb{e}_k)'\sigma_t^k(\pmb{\theta}_k,\pmb{e}_k)|^2dt<\infty,  \ \ \text{a.s.,}\ \  k=0,\dotso,n,$
\item[(b)] the family $\displaystyle \Big\{ U(\mathfrak{X}_{\tau}^k(\pmb{\theta}_k,\pmb{e}_k)):\ \tau \text{ is any }\mathbb{F}-\text{stopping time valued in }[\theta_k,T]\Big\}$ is uniformly integrable, i.e., $U(\mathfrak{X}^k(\pmb{\theta}_k,\pmb{e}_k))$ is of class (D), for $k=0,\dotso,n$,
\item[(c)] $\displaystyle \mathbb{E}\left[\int_{\theta_k}^T\int_E(-U)\left(\mathfrak{X}_s^k(\pmb{\theta}_k,\pmb{e}_k)+\pi_s^k(\pmb{\theta}_k,\pmb{e}_k)'\gamma_s^k(\pmb{\theta}_k,\pmb{e}_k,e)\right)\eta(de)ds\right]<\infty$, \ for $k=0,$ $\dotso, n-1.$
\end{itemize}
\end{definition}
\begin{remark}
Unlike in Definition 4.1, the admissible trading strategy here is in fact independent of stopping times. This is because the investor cannot choose when to stop.  
\end{remark}

\subsubsection{Backward recursive system of RBSDEs}
We decompose $\mathfrak{V}^0$ in \eqref{3-h} into a backward induction as before:
\begin{eqnarray}
\label{5-a}&&\mathfrak{\mathfrak{V}}^n(x,\pmb{\theta}_n,\pmb{e}_n)= \esssup_{\pi^n\in\mathcal{A}^n}  \essinf_{\tau^n\in\mathcal{T}^n}  \mathbb{E}\Big[U(\mathfrak{X}^{n,x}_{\tau^n}-\mathcal{H}_{\tau^n}^n)|\mathcal{F}_{\theta_n}\Big],\ (\pmb\theta_n,\pmb e_n)\in\Delta_n(T)\times E^n,\\
\label{5-b}&&\mathfrak{\mathfrak{V}}^k(x,\pmb{\theta}_k,\pmb{e}_k)=\esssup_{\pi^k\in\mathcal{A}^k} \essinf_{\tau^n\in\mathcal{T}^k}  \mathbb{E}\Big[U(\mathfrak{X}^{k,x}_{\tau^k}-\mathcal{H}_{\tau^k}^k)\\
&&\hspace{1cm}+\int_{\theta_k}^{\tau^k}\int_E\mathfrak{\mathfrak{V}}^{k+1}\Big(\mathfrak{X}_{\theta_{k+1}}^{k,x}+\pi_{\theta_{k+1}}^k\cdot\gamma_{\theta_{k+1}}^k(e_{k+1}),\pmb{\theta}_{k+1},\pmb{e}_{k+1}\Big)\eta(de_{k+1})d\theta_{k+1}|\mathcal{F}_{\theta_k}\Big],\notag
\end{eqnarray}
$(\pmb\theta_k,\pmb e_k)\in\Delta_k(T)\times E^k$, for $k=0,\dotso,n-1$, where 
$$\mathcal{H}^k=R^k+\frac{1}{p}\ln\alpha^k,\ \ k=0,\dotso,n.$$

Consider
\begin{equation}\label{3c}\mathfrak{V}^k(x,\pmb{\theta}_k,\pmb{e}_k)=U\big(x-\mathcal{Y}^k_{\theta_k}(\pmb{\theta}_k,\pmb{e}_k)\big),\text{{ }{ }}k=0,\dotso,n,\notag\end{equation}
where $\{\mathcal{Y}^k_t(\pmb{\theta}_k,\pmb{e}_k)\}_{k=0}^n$ satisfies the RBSDE \textbf{EQ}$(\mathfrak{\mathcal{H}}^k(\pmb\theta_k,\pmb e_k),\mathfrak{f}^k)_{\theta_k\leq t\leq T}$, with $\mathfrak{f}^k$ defined as
\begin{eqnarray}\label{4-ii} \mathfrak{f}^k(t,y,z,\pmb{\theta}_k,\pmb{e}_k)=\inf_{\pi\in A^k}\mathfrak{g}^k(\pi,t,y,z,\pmb{\theta}_k,\pmb{e}_k),\notag\end{eqnarray}
where
\begin{eqnarray}
\label{2vvv} &&\mathfrak{g}^k(\pi,t,y,z,\pmb{\theta}_k,\pmb{e}_k)=\frac{p}{2}\left| z-\sigma_t^k(\pmb{\theta}_k,\pmb{e}_k)'\pi\right|^2-b_t^k(\pmb{\theta}_k,\pmb{e}_k)'\pi\notag\\
&&\hspace{1.8cm}+\frac{1}{p}U(y)\displaystyle\int_EU\left(\pi\cdot\gamma_t^k(\pmb{\theta}_k,\pmb{e}_k,e)-\mathcal{Y}_t^{k+1}(\pmb{\theta}_k,t,\pmb{e}_k,e)\right)\eta(de)\notag
\end{eqnarray}
for $k=0,\dotso,n-1$, and 
$$\label{2www}\mathfrak{g}^n(\pi,t,y,z,\pmb{\theta}_n,\pmb{e}_n)=\frac{p}{2}\left| z-\sigma_t^n(\pmb{\theta}_n,\pmb{e}_n)'\pi\right|^2-b_t^n(\pmb{\theta}_n,\pmb{e}_n)'\pi.$$

\subsubsection{Existence to the recursive system of RBSDEs}
We will make the same boundedness assumption as \textbf{(HB)} in Section~\ref{sec:exsrcsys} except that we will replace $H^k$ with $\mathcal{H}^k$. Let us denote this assumption by (\textbf{HB'}).
\begin{theorem}
Under \textbf{(HB')}, there exists a solution $(\mathcal{Y}^k,\mathcal{Z}^k,\mathcal{K}^k)_{k=0}^n\in\mathop{\prod}\limits_{k=0}^n\mathcal{S}_c^\infty(\Delta_k(T),E^k)\times\mathbf{\bf L}_W^2$ $(\Delta_k(T),E^k)\times\mathbf{A}(\Delta_k(T),E^k) $ to the recursive system of indexed RBSDEs \textbf{EQ}$(\mathcal{H}^k,\mathfrak{f}^k),\ k=0,\dotso,n$.
\end{theorem}
\begin{proof} We prove the result by a backward induction on $k=0,\dotso,n$

For $k=n$. Using the same argument as in the proof of \thref{theorem 6-3},  we can show that there exists a solution $(\mathcal{Y}^n,\mathcal{Z}^n,\mathcal{K}^n)\in\mathcal{S}_c^\infty(\Delta_n(T),E^n)\times\mathbf{\bf L}_W^2(\Delta_n(T),E^n)\times\mathbf{A}(\Delta_n(T),E^n)$ to \textbf{EQ}$(\mathcal{H}^n,\mathfrak{f}^n)$.
  
For $k\in\{0,1,\dotso,n-1\}$. Assume there exists $(\mathcal{Y}^{k+1},\mathcal{Z}^{k+1},\mathcal{K}^{k+1})\in\mathcal{S}_c^\infty(\Delta_{k+1}(T),E^{k+1})\times\mathbf{\bf L}_W^2$ $(\Delta_{k+1}(T),E^{k+1})\times\mathbf{A}(\Delta_{k+1}(T),E^{k+1})$ satisfying \textbf{EQ}$(\mathcal{H}^{k+1},\mathfrak{f}^{k+1})$. Consider the truncated generator
$$\mathfrak{f}^{k,N}(t,y,z,\pmb{\theta}_k,\pmb{e}_k)=\inf_{\pi\in A^k}\mathfrak{g}^k(\pi,t,-N\vee y,z,\pmb{\theta}_k,\pmb{e}_k).$$
Then there exists some constant $C_N>0$, independent of $(\pmb\theta_k,\pmb e_k)$, such that $|\mathfrak{f}^{k,N}|\leq C_N(1+z^2)$. Hence, there exists a solution $(\mathcal{Y}^{k,N}, \mathcal{Z}^{k,N},\mathcal{K}^{k,N})\in\mathcal{S}_c^\infty(\Delta_k(T),E^k)\times\mathbf{\bf L}_W^2(\Delta_k(T),E^k)\times\mathbf{A}(\Delta_k(T),E^k)$ to \textbf{EQ}$(\mathcal{H}^k,\mathfrak{f}^{k,N})$. By Assumption \textbf{(HB')}, $\mathcal{Y}^{k,N}\geq \mathcal{H}^k\geq -C$, where $C>0$ is a constant independent of $N$ and $(\pmb\theta_k,\pmb e_k)$. Therefore, for $N$ large enough, $(\mathcal{Y}^{k,N},\mathcal{Z}^{k,N},\mathcal{K}^{k,N})$ also solves \textbf{EQ}$(\mathcal{H}^k,\mathfrak{f}^k)$.
\end{proof}

\subsubsection{RBSDE characterization by verification theorem}
\begin{theorem}\label{theorem 4-6}
The value functions $(\mathfrak{V}^k)_{k=0}^n$ defined in \eqref{5-a} and \eqref{5-b}, are given by
\begin{equation}\label{3i} \mathfrak{V}^k(x,\pmb{\theta}_k,\pmb{e}_k)=U(x-\mathcal{Y}_t^k(\pmb{\theta}_k,\pmb{e}_k)),\end{equation}
for $\forall x\in\mathbb{R},\ (\pmb{\theta}_k,\pmb{e}_k)\in\Delta_k\times E^k$, where $(\mathcal{Y}^{k},\mathcal{Z}^{k},\mathcal{K}^{k})_{k=0}^n\in\mathop{\prod}\limits_{k=0}^{n}\mathcal{S}_c^\infty$ $(\Delta_{k}(T),E^{k})\times\mathbf{\bf L}_W^2(\Delta_{k}(T),E^{k})\times\mathbf{A}(\Delta_{k}(T),E^{k})$ is a solution of the system of RBSDEs \textbf{EQ}$(\mathcal{H}^k,\mathfrak{f}^k),\ k=0,\dotso,n$. Moreover, there exists a saddle point $(\pi,\tau)\sim(\hat\pi^k,\hat\tau^k)_{k=0}^n$ described by:
\begin{equation}\label{2sss}\hat\pi_t^k(\pmb{\theta}_k,\pmb{e}_k)\in\argmin_{\pi\in A^k}\mathfrak{g}^k\big(\pi,t,\mathcal{Y}_t^k(\pmb\theta_k,\pmb e_k),\mathcal{Z}_t^k(\pmb\theta_k,\pmb e_k),\pmb{\theta}_k,\pmb{e}_k\big),\notag\end{equation}
for $t\in[\theta_k,T]$, and
\begin{equation}\label{3l}\hat\tau^k(\pmb{\theta}_k,\pmb{e}_k):=\inf\left\{t\geq \theta_k: \ \mathcal{Y}_t^k(\pmb{\theta}_k,\pmb{e}_k)=\mathcal{H}_t^k(\pmb{\theta}_k,\pmb{e}_k)\right\},\end{equation}
for $(\pmb{\theta}_k,\pmb{e}_k)\in\Delta_k(T)\times E^k$, a.s., $k=0,\dotso,n$. More specifically, for any admissible trading strategy $(\pi,\tau)\sim(\pi^k,\tau^k)_{k=0}^n$, 
$$\mathbb{E}\left[U(\mathfrak{X}_{\hat\tau^n}^{n,x}-\mathcal{H}_{\hat\tau^n}^n)\big|\mathcal{F}_{\theta_n}\right]\leq\mathbb{E}\left[U(\hat{\mathfrak{X}}_{\hat\tau^n}^{n,x}-\mathcal{H}_{\hat\tau^n}^n)\big|\mathcal{F}_{\theta_n}\right]\leq\mathbb{E}\left[U(\hat{\mathfrak{X}}_{\tau^n}^{n,x}-\mathcal{H}_{\tau^n}^n)\big|\mathcal{F}_{\theta_n}\right],$$
and similar inequalities hold for $k=0,\dotso,n-1$, where $\hat{\mathfrak{X}}^{k,x}$ is the wealth process under control $\hat\pi^k,\ k=0,\dotso,n$.
\end{theorem}
\begin{proof}
We follow the steps in the proof of \thref{theorem 4-4}. \\
\textbf{Step 1:} \textbf{We will show for $(\pmb\theta_k,\pmb e_k)\in\Delta_k(T)\times E^k$,}
\begin{equation}\label{3m}U(x-\mathcal{Y}^k_{\theta_k}(\pmb{\theta}_k,\pmb{e}_k,\nu^k))\geq \mathfrak{V}^k(x,\pmb{\theta}_k,\pmb{e}_k),\ \ \ k=0,\dotso,n.\end{equation}
Let  $(\mathcal{Y}^{k},\mathcal{Z}^{k},\mathcal{K}^{k})\in\mathcal{S}_c^\infty(\Delta_{k}(T),E^{k})\times\mathbf{\bf L}_W^2(\Delta_{k}(T),E^{k})\times\mathbf{A}(\Delta_{k}(T),E^{k})$ be a solution of the RBSDE system. For $\nu^k\in\mathcal{A}^k$, \ $\forall x\in\mathbb{R}, \ (\pmb{\theta}_k,\pmb{e}_k)\in\Delta_k(T)\times E^k$, define $(\upxi^k)_{k=0}^n$ as:
\begin{eqnarray}
\label{5-aa}\upxi_t^k(x,\pmb{\theta}_k,\pmb{e}_k,\nu^k)&:=&U\big(\mathfrak{X}_t^{k,x}-\mathcal{Y}_t^k(\pmb{\theta}_k,\pmb{e}_k)\big)\\
&+&\int_{\theta_k}^t\int_EU\left(\mathfrak{X}_r^{k,x}+\nu_r^k\cdot\gamma_r^k(\pmb{\theta}_{k},e_{k},e)-\mathcal{Y}_r^{k+1}(\pmb{e}_{k},r,\pmb{e}_k,e)\right)\eta(de)dr,\notag
\end{eqnarray}
for $k=0,\dotso,n-1$, and
\begin{equation}\label{5-bb}\upxi_t^n(x,\pmb{\theta}_n,\pmb{e}_n,\nu^n):=U\big(\mathfrak{X}_t^{n,x}-\mathcal{Y}_t^n(\pmb{\theta}_n,\pmb{e}_n)\big).\end{equation}
\noindent Applying It$\hat{\text{o}}$'s formula, we obtain, for $k=0,\dotso,n$,
\begin{eqnarray}
\label{3n}&&d\upxi_t^k(x,\pmb{\theta}_k,\pmb{e}_k,\nu^n)=pU(\mathfrak{X}_t^{k,x}-\mathcal{Y}_t^k(\pmb{\theta}_k,\pmb{e}_k)) \Big[\big(-\mathfrak{f}^k(t,\mathcal{Y}_t^k,\mathcal{Z}_t^k,\pmb{\theta}_k,\pmb{e}_k)\notag\\
&&\hspace{2cm}+\mathfrak{g}^k(\nu_t^k,t,\mathcal{Y}_t^k,\mathcal{Z}_t^k,\pmb{\theta}_k,\pmb{e}_k)\big)dt-d\mathcal{K}_t^k(\pmb{\theta}_k,\pmb{e}_k)+(\mathcal{Z}_t^k-\sigma_t^k(\pmb{\theta}_k,\pmb{e}_k)'\nu_t^k)\cdot dW_t\Big],\notag
\end{eqnarray}
Define $\hat\tau^k$ as in \eqref{3l}, then $d\mathcal{K}_{t\wedge\hat\tau^k}^k=0, \ \theta_k\leq t\leq T$. Therefore, $(\upxi^k_{t\wedge\hat\tau^k})_{\theta_k\leq t\leq T}$ is a local super-martingale. By introducing a localizing sequence of stopping times $(\rho_m)_m$, and then letting $m\rightarrow\infty$, we can show for $k=0,\dotso,n,$
\begin{equation}\label{3o}\upxi^k_{t\wedge\tilde\tau^k}\geq\mathbb{E}\left[\upxi^k_{s\wedge\tilde\tau^k}\Big|\mathcal{F}_t\right], \ \ \theta_k\leq t\leq s\leq T.\notag\end{equation}
In particular,
\begin{equation}\label{5ab}U\big(x-\mathcal{Y}^n_{\theta_n}(\pmb{\theta}_n,\pmb{e}_n)\big)=\upxi^n_{\theta_n}\geq\mathbb{E}\left[\upxi^n_{\tilde\tau^n}\Big|\mathcal{F}_{\theta_n}\right]=\mathbb{E}\Big[U(\mathfrak{X}^{n,x}_{\tilde\tau^n}-\mathcal{H}^n_{\tilde\tau^n})\Big|\mathcal{F}_{\theta_n}\Big].\end{equation}
Hence, 
\begin{equation}\label{3q}U\big(x-\mathcal{Y}^n_{\theta_k}(\pmb{\theta}_n,\pmb{e}_n)\big)\geq \essinf_{\tau^n\in\mathcal{T}^n} \mathbb{E}[U(\mathfrak{X}^{n,x}_{\tau^n}-\mathcal{H}^n_{\tau^n})|\mathcal{F}_{\theta_k}].\notag\end{equation}
for any $\nu^n\in\mathcal{A}^n$. So \eqref{3m} follows for $k=n$. Similarly, it holds for $k=0,\dotso,n-1$.
\medskip

\noindent\textbf{Steps 2\& 3:} Similar to the proof of \thref{theorem 4-4}. 
\medskip

\noindent\textbf{Step 4:} \textbf{We will show \eqref{3i} holds and $(\pi,\tau)\sim(\hat\pi^k,\hat\tau^k)_{k=0}^n$ is a saddle point.} Under the admissible control $\hat\pi^k$, the dynamics of ($\upxi^k)_k$ defined in \eqref{5-aa} and \eqref{5-bb} are given by 
\begin{equation}\label{3r}d\upxi_t^k(x,\pmb{\theta},\pmb{e},\hat\pi^k)=pU(\mathfrak{X}_t^{k,x}-\mathcal{Y}_t^k)\Big[-d\mathcal{K}_t^k(\pmb{\theta}_k,\pmb{e}_k)+(\mathcal{Z}_t^k-\sigma_t^k(\pmb{\theta}_k,\pmb{e}_k)'\nu_t^k)\cdot dW_t\Big],\notag\end{equation}
for $k=0,\dotso,n$. By the uniform integrality of $\upxi_t^k$, we know $\upxi_t^k$ is a sub-martingale. Consider when $k=n$.  For any $\mathbb{F}$-stopping time $\tau^n$ valued in $[\theta_n,T]$,
\begin{equation}\label{3u}U(x-\mathcal{Y}^n_{\theta_n})\leq \mathbb{E}[U(\hat{\mathfrak{X}}^{n,x}_{\tau^n}-\mathcal{Y}^n_{\tau^n})|\mathcal{F}_{\theta_n}]\leq\mathbb{E}[U(\hat{\mathfrak{X}}^{n,x}_{\tau^n}-\mathcal{H}^n_{\tau^n})|\mathcal{F}_{\theta_n}],\end{equation}
Therefore, we have
\begin{equation}\label{3s}U(x-\mathcal{Y}_{\theta_n}^n)\leq\essinf_{\tau^n\in\mathcal{T}^n}\mathbb{E}\Big[U(\hat{\mathfrak{X}}_{\tau^n}^{n,x}-\mathcal{H}_{\tau^n}^n)\big|\mathcal{F}_{\theta_n}\Big]\leq\esssup_{\pi^n\in\mathcal{A}^n}\essinf_{\tau^n\in\mathcal{T}^n}\mathbb{E}\Big[U(\mathfrak{X}_{\tau^n}^{n,x}-\mathcal{H}_{\tau^n}^n)\big|\mathcal{F}_{\theta_n}\Big].\notag\end{equation}
Now, the last equation along with \eqref{3m} implies that \eqref{3i} holds for $k=n$. 

By the definition and admissibility of $\hat\pi^n$, we can show that under control $\hat\pi^n$, $\upxi^n_{t\wedge \hat\tau^n}$ is a martingale. Thus from \eqref{3u} we have
\begin{equation}\label{3t} \begin{split}\mathbb{E}\left[U(\hat{\mathfrak{X}}_{\hat\tau^n}^{n,x}-\mathcal{H}_{\hat\tau^n}^n)\big|\mathcal{F}_{\theta_n}\right]=\mathbb{E}\left[U(\hat{\mathfrak{X}}_{\hat\tau^n}^{n,x}-\mathcal{Y}_{\hat\tau^n}^n)\big|\mathcal{F}_{\theta_n}\right]=U(x-\mathcal{Y}_{\theta_n}^n)\\
\leq \mathbb{E}\left[U(\hat{\mathfrak{X}}_{\tau^n}^{n,x}-\mathcal{Y}_{\tau^n}^n)\big|\mathcal{F}_{\theta_n}\right]\leq\mathbb{E}\left[U(\hat{\mathfrak{X}}_{\tau^n}^{n,x}-\mathcal{H}_{\tau^n}^n)\big|\mathcal{F}_{\theta_n}\right].\end{split}\notag\end{equation}
And from \eqref{5ab} we have
\begin{equation}\label{v}\begin{split} \mathbb{E}\left[U(\hat{\mathfrak{X}}_{\hat\tau^n}^{n,x}-\mathcal{H}_{\hat\tau^n}^n)\big|\mathcal{F}_{\theta_n}\right]=\mathbb{E}\left[U(\hat{\mathfrak{X}}_{\hat\tau^n}^{n,x}-\mathcal{Y}_{\hat\tau^n}^n)\big|\mathcal{F}_{\theta_n}\right]=U(x-\mathcal{Y}_{\theta_n}^n)\\
\geq\mathbb{E}\left[U(\mathfrak{X}_{\hat\tau^n}^{n,x}-\mathcal{Y}_{\hat\tau^n}^n)\big|\mathcal{F}_{\theta_n}\right]=\mathbb{E}\left[U(\mathfrak{X}_{\hat\tau^n}^{n,x}-\mathcal{H}_{\hat\tau^n}^n)\big|\mathcal{F}_{\theta_n}\right] \end{split}\notag\end{equation}
Thus, $(\hat\pi^n,\hat\tau^n)$ is a saddle point. Similarly, it can be shown that the corresponding conclusions hold for $k=0,\dotso,n-1$ using \eqref{5-b}.
\end{proof}

\bibliographystyle{siam}
\bibliography{ref}

\begin{thebibliography}{10}

\bibitem{MR509204}
{\sc M.~T. Barlow}, {\em Study of a filtration expanded to include an honest
  time}, Z. Wahrsch. Verw. Gebiete, 44 (1978), pp.~307--323.

\bibitem{erhan3}
{\sc E.~Bayraktar}, {\em A proof of the smoothness of the finite time horizon
  {A}merican put option for jump diffusions}, SIAM J. Control Optim., 48
  (2009), pp.~551--572.

\bibitem{MR2260062}
{\sc E.~Bayraktar, S.~Dayanik, and I.~Karatzas}, {\em Adaptive {P}oisson
  disorder problem}, Ann. Appl. Probab., 16 (2006), pp.~1190--1261.

\bibitem{MR2928345}
{\sc E.~Bayraktar, I.~Karatzas, and S.~Yao}, {\em Optimal stopping for dynamic
  convex risk measures}, Illinois J. Math., 54 (2010), pp.~1025--1067 (2012).

\bibitem{MR636252}
{\sc P.~Br{\'e}maud}, {\em Point processes and queues}, Springer-Verlag, New
  York, 1981.
\newblock Martingale dynamics, Springer Series in Statistics.

\bibitem{texterhan}
{\sc E.~{\c{C}}{\i}nlar}, {\em Probability and stochastics}, vol.~261 of
  Graduate Texts in Mathematics, Springer, New York, 2011.

\bibitem{MR1283589}
{\sc M.~H.~A. Davis}, {\em Markov models and optimization}, vol.~49 of
  Monographs on Statistics and Applied Probability, Chapman \& Hall, London,
  1993.

\bibitem{Schweizer1}
{\sc F.~Delbaen, P.~Grandits, T.~Rheinl{\"a}nder, D.~Samperi, M.~Schweizer, and
  C.~Stricker}, {\em Exponential hedging and entropic penalties}, Math.
  Finance, 12 (2002), pp.~99--123.

\bibitem{El2}
{\sc N.~El~Karoui, M.~Jeanblanc, and Y.~Jiao}, {\em What happens after a
  default: the conditional density approach}, Stochastic Process. Appl., 120
  (2010), pp.~1011--1032.

\bibitem{El1}
{\sc N.~El~Karoui, C.~Kapoudjian, E.~Pardoux, S.~Peng, and M.~C. Quenez}, {\em
  Reflected solutions of backward {SDE}'s, and related obstacle problems for
  {PDE}'s}, Ann. Probab., 25 (1997), pp.~702--737.

\bibitem{Ying1}
{\sc Y.~Hu, P.~Imkeller, and M.~M{\"u}ller}, {\em Utility maximization in
  incomplete markets}, Ann. Appl. Probab., 15 (2005), pp.~1691--1712.

\bibitem{Kharroubi1}
{\sc K.~Idris and T.~Lim}, {\em Progressive enlargement of filtrations and
  backward sdes with jumps},  (2012).
\newblock To appear in the Journal of Theoretical Probability, arXiv:1101.2815.

\bibitem{2012arXiv1201.2690J}
{\sc M.~{Jeanblanc}, A.~{Matoussi}, and A.~{Ngoupeyou}}, {\em {Robust utility
  maximization in a discontinuous filtration}}, ArXiv e-prints,  (2012).

\bibitem{Jeulin1}
{\sc T.~Jeulin}, {\em Semi-martingales et grossissement d'une filtration},
  vol.~833 of Lecture Notes in Mathematics, Springer, Berlin, 1980.

\bibitem{MR519998}
{\sc T.~Jeulin and M.~Yor}, {\em Grossissement d'une filtration et
  semi-martingales: formules explicites}, in S\'eminaire de {P}robabilit\'es,
  {XII} ({U}niv. {S}trasbourg, {S}trasbourg, 1976/1977), vol.~649 of Lecture
  Notes in Math., Springer, Berlin, 1978, pp.~78--97.

\bibitem{pham2}
{\sc Y.~Jiao, I.~Kharroubi, and H.~Pham}, {\em Optimal investment under
  multiple defaults risk: a {BSDE}-decomposition approach}, Annals of Applied
  Probability, 23 (2) (2013), pp.~455--491.

\bibitem{Pham4}
{\sc Y.~Jiao and H.~Pham}, {\em Optimal investment with counterparty risk: a
  default-density model approach}, Finance Stoch., 15 (2011), pp.~725--753.

\bibitem{Karatzas1}
{\sc I.~Karatzas and S.~G. Kou}, {\em Hedging {A}merican contingent claims with
  constrained portfolios}, Finance Stoch., 2 (1998), pp.~215--258.

\bibitem{MR2178425}
{\sc I.~Karatzas and I.-M. Zamfirescu}, {\em Game approach to the optimal
  stopping problem}, Stochastics, 77 (2005), pp.~401--435.

\bibitem{MR2172784}
\leavevmode\vrule height 2pt depth -1.6pt width 23pt, {\em Martingale approach
  to stochastic control with discretionary stopping}, Appl. Math. Optim., 53
  (2006), pp.~163--184.

\bibitem{MR2435857}
\leavevmode\vrule height 2pt depth -1.6pt width 23pt, {\em Martingale approach
  to stochastic differential games of control and stopping}, Ann. Probab., 36
  (2008), pp.~1495--1527.

\bibitem{Lim1}
{\sc I.~Kharroubi and T.~Lim}, {\em Progressive enlargement of filtrations and
  backward sdes with jumps},  (2012).
\newblock 2012, preprint arXiv:1101.2815.

\bibitem{Kobylanski1}
{\sc M.~Kobylanski, J.~P. Lepeltier, M.~C. Quenez, and S.~Torres}, {\em
  Reflected {BSDE} with superlinear quadratic coefficient}, Probab. Math.
  Statist., 22 (2002), pp.~51--83.

\bibitem{Sircar1}
{\sc T.~Leung and R.~Sircar}, {\em Exponential hedging with optimal stopping
  and application to employee stock option valuation}, SIAM J. Control Optim.,
  48 (2009), pp.~1422--1451.

\bibitem{pham1}
{\sc H.~Pham}, {\em Stochastic control under progressive enlargement of
  filtrations and applications to multiple defaults risk management},
  Stochastic Process. Appl., 120 (2010), pp.~1795--1820.

\bibitem{MR2020294}
{\sc P.~E. Protter}, {\em Stochastic integration and differential equations},
  vol.~21 of Applications of Mathematics (New York), Springer-Verlag, Berlin,
  second~ed., 2004.
\newblock Stochastic Modelling and Applied Probability.

\bibitem{Touzi1}
{\sc H.~M. Soner and N.~Touzi}, {\em Dynamic programming for stochastic target
  problems and geometric flows}, J. Eur. Math. Soc. (JEMS), 4 (2002),
  pp.~201--236.

\bibitem{Song1}
{\sc S.~Song}, {\em Optional splitting formula in a progressively enlarged
  filtration},  (2012).
\newblock arXiv:1208.4149.

\bibitem{Wagner1}
{\sc D.~H. Wagner}, {\em Survey of measurable selection theorems: an update},
  in Measure theory, {O}berwolfach 1979 ({P}roc. {C}onf., {O}berwolfach, 1979),
  vol.~794 of Lecture Notes in Math., Springer, Berlin, 1980, pp.~176--219.

\end{thebibliography}

\end{document}